\documentclass[11pt]{article}

\usepackage{amsfonts}
\usepackage{amsmath,amsthm,amscd,amssymb,mathrsfs,setspace}
\usepackage{latexsym,epsf,epsfig}
\usepackage{cite}
\usepackage{color}
\usepackage[hmargin=2.6cm,vmargin=2.6cm]{geometry}
\usepackage{hyperref}
\newcommand{\ds}{\displaystyle}

\newcommand{\reals}{\mathbb{R}}

\newcommand{\cD}{\mathscr{D}}

\newcommand{\cE}{{\mathcal{E}}}

\newcommand{\mN}{\mathbb{N}}

\newcommand{\bu}{\mathbf u}

\newcommand{\bw}{\mathbf w}
\newcommand{\bF}{\mathbf F}

\newcommand{\bV}{\mathbf V}

\newcommand{\grad}{\nabla}
\newcommand{\Om}{\Omega}

\theoremstyle{plain}
\newtheorem{theorem}{Theorem}[section]
\newtheorem{lemma}[theorem]{Lemma}
\newtheorem{proposition}[theorem]{Proposition}

\newtheorem{corollary}[theorem]{Corollary}
\newtheorem{assumption}{Assumption}
\theoremstyle{remark}
\newtheorem{definition}{Definition}
\newtheorem{remark}{Remark}[section]
\numberwithin{equation}{section}
\numberwithin{theorem}{section}
\numberwithin{remark}{section}
\numberwithin{assumption}{section}
\numberwithin{condition}{section}

\begin{document}

\title{Nonlinear Quasi-static Poroelasticity}
\author{Lorena Bociu\footnote{2311 Stinson Dr., North Carolina State University, Raleigh, NC, 27607; {\em lvbociu@ncstate.edu}}  \hskip 2cm Justin T. Webster\footnote{1000 Hilltop Dr., University of Maryland, Baltimore County, Baltimore, MD, 21250;~ {\em websterj@umbc.edu}} }

\maketitle

\begin{abstract}
\noindent  We analyze a quasi-static Biot system of poroelasticity for both compressible and incompressible constituents. The main feature of this model is a nonlinear coupling of pressure and dilation through the system's permeability tensor. Such a model has been analyzed previously from the point of view of constructing weak solutions through a fully discretized approach. In this treatment, we consider simplified Dirichlet type boundary conditions in both the elastic displacement and pressure variables and give a full treatment of weak solutions. Our construction of weak solutions for the nonlinear problem is based on a priori estimates, a requisite feature in addressing the nonlinearity. We utilize a spatial semi-discretization and employ a multi-valued fixed point argument for a clear construction of weak solutions. We also provide regularity criteria for uniqueness of solutions. 
\vskip.25cm

\vskip.25cm
\noindent
{\em 2010 AMS}: 74F10, 76S05, 35M13, 35A01, 35B65, 35Q86, 35Q92
\vskip.4cm
\noindent Acknowledged Support: L. Bociu was partially supported by NSF-DMS 1555062 (CAREER). J.T. Webster was partially supported by NSF-DMS 1907620.
\end{abstract}

\section{Introduction}

Poroelasticity refers to (Darcy) fluid flow within a deformable, porous medium. The development of this field has been inspired by geophysics and petroleum engineering problems, in particular, reservoir, environmental, and earthquake engineering. Mathematically, the subject was initiated in  the 1D work of Terzaghi in the 1920s, and the groundbreaking consolidation theory developed by Biot in the 1940--50s \cite{biot}.  It was Biot's work which instigated the rapid development and progress of this field. The relevant literature is now abundant, and we only list here representative fundamental treatments: \cite{barry, dcheng, coussy, philwheel, frenchpaper, zenisek, show1, cao}.  
In all of the works motivated by geophysical applications, the poroelastic structures considered are  soil and/or rock  (for instance, in most of the aforementioned references). However, cartilages, bones, as well as brain, heart, and liver tissue etc., are also poroelastic structures. Therefore, the theory of poroelasticity can be used and applied to fluid flows inside cartilages, bones, and engineered tissue scaffolds, as well as in perfusion in the optic nerve head---see \cite{bgsw, MBE, MBE2} and references and discussion therein.

From a mathematical point of view, poroelastic systems constitute a coupled system of a (possibly degenerate) parabolic fluid pressure and a hyperbolic (inertial) or elliptic (quasi-static) system of elasticity for the displacements of the porous matrix containing the fluid. The saturated elastic matrix is modeled through homogenization \cite{frenchpaper,biot}, in the sense that the pressure and displacement are distributed quantities throughout the physical domain. In this treatment, we focus on poroelastic models with specific applications in biomechanics (in contrast to those tailored to geomechanical systems). Thus we work under the assumptions of {\em full saturation, negligible inertia, small deformations, and (possibly) compressible mixture components}. The applications of interest give rise to a permeability taken as a nonlinear function of the so called fluid content (a particular linear combination of pressure and dilation). This type of nonlinear coupling introduces a variety of complications detailed below, and, in particular, destroys the monotone nature of the problem. 

Such a nonlinear poroelastic model was first considered---from a mathematical point of view---in \cite{cao}, and shortly after in \cite{bgsw}.\footnote{We note that in \cite{show2} a nonlinear version of the Biot problem is considered, but the structure of the nonlinearity there is monotone in nature and different from the physical nonlinearity presented here for biological applications.} The former reference \cite{cao} and sequel \cite{cao2} focus on the {\it compressible Biot model} and construct  
weak solutions through a full spatio-temporal discretization,  in the mathematically simplified framework of homogeneous Dirichlet boundary conditions for both fluid pressure and solid displacement.  Those references take the earlier linear theory \cite{frenchpaper, show1,showmono} as their primary motivation, and use Brouwer's fixed point theorem at the level of the fully discretized problem. The latter reference \cite{bgsw} focuses on {\it Biot models with incompressible constituents} and constructs weak solutions (also using discretizations in both time and space \cite{zenisek}) for both poroelastic and poro-viscoelastic systems with non-homogeneous, mixed boundary conditions that are  physically relevant to opthalmological applications. A key theme in this latter work \cite{bgsw} is the careful analysis of the requisite boundary and source regularity for the construction of weak solutions, as this aspect is crucial in understanding the
mechanisms leading to tissue damage in the optic nerve head, and consequent vision loss possibly associated
with glaucoma. 
Both \cite{cao,bgsw} obtain a priori estimates in the fully discretized setting, and much of the challenge lies in adequately addressing the 
nonlinear and non-monotone coupling to obtain a weak solution in the limit. The reference \cite{cao} provides a straightforward regularity criterion for uniqueness of solutions, but does not actually consider smooth solutions, nor address the permissibility of multipliers used to obtain estimates.

In this treatment, we provide  a careful mathematical construction of weak solutions using semi-discretization in space, in the setting of fully homogeneous boundary conditions. 
One primary goal is to clearly elucidate the challenges introduced into the Biot problem by the inclusion of non-monotone nonlinear coupling. We also include a novel, sharper uniqueness criterion for solutions of sufficient smoothness. Our approach is based on a priori estimates for the time-dependent linearization, from which we construct a fixed point correspondence. An interesting feature of this approach is that we cannot appeal to uniqueness of solutions for the aforementioned linear problem, as we do not satisfy requisite hypotheses for established theories (e.g., that in \cite{showmono}). Indeed, weak solutions themselves are not permissible test functions, presenting a great hurdle in the analysis. To address this issue, we utilize a multi-valued fixed point approach, along with a careful construction of the correspondence between the permeability function and the resulting fluid content.

{\bf In summary}: This paper addresses the existence of weak solutions to a quasilinear Biot system, based on a fixed point approach that circumvents the lack of monotonicity in the system's nonlinear coupling.
We believe that the construction given here is quite natural, and illustrative of the complexities in the analysis that are introduced by the presence of nonlinearity and its interaction with the boundary conditions.

\subsection{PDE Model}

Let $\Omega$ be an open, bounded subset of $\mathbb{R}^3$ representing the spatial domain occupied by the (fully saturate) fluid-solid mixture, with smooth boundary $\Gamma = \partial \Omega$. Let $\mathbf{x}$ be the position vector of each point in the body with respect to a fixed Cartesian reference frame. The symbol $\mathbf n$ will be used to denote the unit outward normal vector to $\Omega$. 
Let 
 $V_f(\mathbf{x},t)$ be the volume occupied by the fluid component in a representative  volume $V(\mathbf{x},t)$ element centered at $\mathbf{x} \in \Omega$ at time $t$. Then the porosity $\phi$ and the fluid content $\zeta$ are given by 
$\ds \phi(\mathbf{x},t) = V_f(\mathbf{x},t)/V(\mathbf{x},t)
\quad \mbox{and} \quad 
\zeta(\mathbf{x},t) = \phi(\mathbf{x},t)-\phi_0(\mathbf{x})$, 
where $\phi_0$ is a baseline (local) value for the porosity.
\vskip.1cm

\noindent \underline{Balance Equations:} Under the assumptions of small deformations, full saturation of the mixture,
and negligible inertia, we can write
the balance of linear momentum for the mixture and the balance of mass for the fluid component as
\begin{equation}\label{eq:system_1}
\partial_t\zeta + \nabla\cdot\mathbf{v}=S(\mathbf x, t) 
\quad \mbox{and}\quad 
-\nabla\cdot\mathbf T + \mathbf F(\mathbf x, t) = \mathbf 0 \quad \mbox{in}\; \Omega\times (0,T),
\end{equation}
where $\mathbf T$ is the total stress, 
$\mathbf{v}$ is the discharge velocity (also commonly called the {\em Darcy velocity} \cite{show1}),
$\mathbf{F}$ is a body force per unit of volume, and
$S$ is a net volumetric fluid production rate. \vskip.1cm

\noindent \underline{Constitutive Equations:}
We complement the balance equations with the constitutive equations:
\vskip.1cm
\noindent The {\it total stress} of the mixture is given by
\begin{equation}
 \label{eq:stress}
\mathbf T = \mathbf T_e - \alpha p \mathbf{I} = 2\mu_e \varepsilon(\mathbf{u}) + \lambda_e (\nabla \cdot \mathbf{u})\, \mathbf{I} - \alpha p \mathbf{I}, 
\end{equation}
where $\mathbf{u}$ is the solid displacement, the symmetrized gradient ~
$\varepsilon(\mathbf{u}) = (\nabla \mathbf u + \nabla \mathbf u^T)/2$~ gives the strain tensor,  $\alpha$ is the Biot-Willis constant,
$p$ is the Darcy fluid pressure, $\mathbf{I}$ is the identity tensor, and $\lambda_e$ and $\mu_e$ are the  elasticity parameters.

\vskip.2cm
\noindent The \textit{discharge (Darcy) velocity} has the following formula via Darcy's law \cite{show1}:
\begin{equation}
\label{eq:vel_darcy}
\mathbf{v} = - k(\phi)\mathbf{I} \nabla  p.
\\
\end{equation} 
The particular form of the relationship between the permeability $k$ and the porosity $\phi$ depends on the geometrical architecture of the pores in the elastic matrix and the properties of the fluid. 
We allow for $k$ to be a general continuous function, assuming only that it is bounded above and below (as discussed below, in Assumption \ref{Assumpk}, and consistent with \cite{cao,bgsw}).
\vskip.2cm
\noindent The \textit{fluid content} is given by 
\begin{equation}
\label{eq:fluid_increment}\
 \zeta = c_0p+\alpha \nabla \cdot \mathbf{u},
\end{equation}
where $c_0$ is the constrained specific storage coefficient \cite{dcheng,show1,frenchpaper}. Using the relation between porosity and fluid content, as well as the definition of permeability, we can see that permeability in the system depends nonlinearly on the fluid content. In the special case of incompressible constituents, due to the fact that the constrained storage coefficient $c_0 = 0$ and $\alpha = 1$, the permeability becomes a nonlinear function of dilation alone. This is the scenario that is specifically  addressed in \cite{bgsw}. \vskip.2cm

\noindent \underline{Boundary Conditions:}  We consider homogeneous Dirichlet boundary conditions for both the structural displacement $\bu$  (and hence $\bu_t$, when defined) and the fluid pressure $p$ \begin{equation}
\mathbf u = \mathbf 0, \quad p =0 \quad \mbox{on}\; \Gamma\,.
\end{equation}
This choice is 
in line with the model considered in \cite{cao,owc}. 
In our previous work \cite{bgsw}, we considered complex physical configurations, incorporating both nonhomogeneous Dirichlet and Neumann boundary conditions for the elastic displacement and fluid pressure. These physically motivated mixed boundary conditions could be incorporated here, and this is the subject of future work.
\vskip.2cm
\noindent \underline{Initial Conditions:} Initial conditions are to be specified for the fluid content, $\zeta$, as it is the only term which appears under the temporal integration in the mass-balance equation \eqref{eq:system_1}:
\hskip4cm $\zeta(\mathbf{x}, 0) = d_0  \quad \mbox{in}\;\; \Omega.$

\begin{remark} \label{dataremark} In discussing {\em various notions of solutions} (as in Section \ref{Uniqueness}), one can find the requirement that $d_0=\zeta(0)=[c_0p+\alpha \nabla \cdot \bu](0)$  for some $\bu(t=0)=\bu_0$ specified independently, taken in an appropriate space (see \cite{zenisek}, as well as \cite{bgsw}). In these works, a different construction for solutions is utilized.  We do note that for the linear case, in the most general ``weak" setting \cite{show1}, only $d_0$ should be needed. In this weak situation, the  construction is done independent of a priori estimates obtained in standard Hilbert spaces such as $L^2(\Omega)$ and $H_0^1(\Omega)$. In \cite{owc}, solutions are also constructed in the linear case, but initial data is taken to be smoother than ``finite energy" considerations require. 
\end{remark}

\noindent \underline{PDE System:} 
To summarize, below is the nonlinear system under consideration:
\begin{eqnarray} 
&- \nabla\cdot [2\mu_e \varepsilon(\mathbf{u}) + \lambda_e (\nabla \cdot \mathbf{u})\, \mathbf{I} - \alpha p \mathbf{I}]= \bF \hspace{.4cm} &\text{in} ~~~\Omega\times(0,T) \label{1}\label{Eq:elasticity}\\ 
& \zeta_t-\nabla\cdot \big[k(\zeta)\nabla p\big]=S&\text{in} ~~~\Omega\times(0,T)  \label{2}\label{Eq:pressure}\\
& \zeta = c_0p+\alpha \nabla \cdot \bu & \text{in} ~~~\Omega\times(0,T) \\
& \bu= {\mathbf 0} ~~\text{and} ~~p = 0 &\text{on} ~~~\Gamma \times(0,T)\\
&\zeta(0) = d_0 &\text{in} ~~~\Omega, \mbox{ for } \; t=0
 \label{3}\label{Eq:IC}
\hspace{.5cm}
\end{eqnarray}

Note that \eqref{Eq:elasticity} can be written equivalently  as 
$$-\mu \Delta \mathbf{u} - (\lambda + \mu) \grad(\grad \cdot \mathbf{u}) + \alpha \grad p = \bF,
$$
where the Laplacian above is interpreted component-wise.

\begin{assumption}\label{Assumpk}[Bounds on the Permeability Function]
We assume that the permeability function $k: \reals \to \reals$ is continuous and that there exist constants $k_1 > 0$ and $k_2 >0$ s.t. 
$$0 < k_1 \leq k(x) \leq k_2, \ \ \forall x \in \reals.$$
\end{assumption}
\begin{remark} With a slight abuse of notation, we denote by
$k(\Psi(\cdot,t))$ the Nemytskii operator associated with $k$. Using our assumptions on the function $k$, and the theory of superposition operators \cite{RR, Tr}, we have that the operator $k$ is bounded and continuous from $L^2(\Omega \times (0,T))$ into $L^2(\Omega \times (0,T))$. 

In order to obtain uniqueness of solution in Section \ref{Uniqueness}, we will further assume that $k$ is a globally Lipschitz function, i.e., $k \in Lip(\mathbb R)$. \end{remark}

\begin{assumption}\label{Assumpcoeff}[Other Assumptions] In what follows, for simplicity, we set to unity non-essential (from the mathematical point of view) parameters. This is to say, we take $\lambda_e=\mu_e=\alpha=1$. The parameter $c_0$ is retained as is, with no dependence on other parameters,  as we will take $c_0 \searrow 0$ in the construction of weak solutions for the case of incompressible constituents.
\end{assumption}

\section{Main Results}
\subsection{Notation, Function Spaces, and Conventions}
We make the following conventions for the rest of the paper. Norms $\|\cdot\|_D$ are taken to be $L^2(D)$ for a domain $D$. Inner products in $L^2(D)$ are written as $(\cdot,\cdot)_D$, where the subscript will be omitted when the context is clear.  The standard Sobolev space of order $s$ defined on a domain $D$ \cite{kesavan} will be denoted by $ H^s(D)$, with $H^s_0(D)$ denoting the closure of $C_0^{\infty}(D)$ in the $H^s(D)$ norm
(which we denote by $\|\cdot\|_{H^s(D)}$ or $\|\cdot\|_{s}$).   Vector valued spaces will be denoted as $\mathbf L^2(\Omega) \equiv [L^2(\Omega)]^n$ and $\mathbf H^s(\Omega) = [H^s(\Omega)]^n$. We make use of the standard notation for the trace of functions $\gamma: H^1(D) \to H^{1/2}(\partial D)$ which generalizes restriction to a lower dimensional manifold. We will make use of the spaces $L^2(0,T;U)$ and $H^s(0,T;U)$, where $U$ is a Banach space. These norms (and associated inner products) will be denoted with the appropriate subscript, e.g., $||\cdot||_{L^2(0,T;U)}$. We utilize the Frobenius scalar product for tensors with the Einstein summation convention:
$$(\mathbf A, \mathbf B) = \int_{\Omega} (A_{ij}B_{ij})d\Omega,$$ 
sometimes also denoted by $\int_{\Omega}\mathbf A :\mathbf B~d\Omega$. Notice that, when $\mathbf A=\mathbf B$, we write 
$$(\mathbf A, \mathbf A) = \int_{\Omega} \mathbf A : \mathbf A ~d\Omega=\sum_{i,j}(A_{ij},A_{ij})=||\mathbf A||^2,$$ the latter norm taken in the Frobenius sense.

The primary spaces in our analysis below are
\begin{align}
V \equiv  ~H^1_{0}(\Omega), & \hskip1.5cm
\bV \equiv ~(H^1_{0}(\Omega))^3,\hskip1.5cm 
\end{align} for the pressure $p$ and elastic displacement  $\bu$, respectively. 
The norms in these spaces are taken in the natural sense, respectively, accounting for Poincar\'e's and Korn's inequalities \cite{kesavan}. For $V$, we take the standard gradient norm: $||v||_V = ||v||_{H_0^1(\Omega)} = ||\nabla v||_{L^2(\Omega)}$. We will frequently need to denote the duality pairing between $V$ and $V'$ or $\bV$ and $\bV'$, for which we will use the generic notation $\langle \cdot, \cdot \rangle$. (For more general spaces $B$ and $B'$, we may write $\langle \cdot , \cdot \rangle_{B' \times B}$ for clarity.)

We utilize the notation $\nabla \bu$ as the Jacobian matrix of $\mathbf u$ and the associated symmetric gradient $\varepsilon(\bu)$, yielding the following definitions and
formal identities:
\begin{align}
&\nabla \mathbf u = (\partial_j u^i),~~\nabla\mathbf u^T =(\partial_iu^j);  ~~ \varepsilon(\mathbf u) = \dfrac{1}{2}[\nabla \bu +\nabla \bu^T]&\\
&(\nabla \bu, \nabla \bw)_{\Omega} =  \int_{\Omega} [\nabla \bu : \nabla \bw] d\mathbf x =~  \int_{\Omega} tr (\nabla \bu \nabla\bw^T)d\mathbf x=  \int_{\Omega} tr(\nabla \bw^T \nabla \bu)d\mathbf x=(\nabla\bw^T,\nabla \bu^T)& \\
&(\varepsilon(\bu),\nabla \bw)=\frac{1}{2}(\nabla \bu , \nabla \bw)+\frac{1}{2}(\nabla \bu^T, \nabla \bw)=~ (\varepsilon(\bu), \varepsilon(\bw)) .&
\end{align}

In the simplified setting, the bilinear form associated with the elasticity operator is given by
\begin{equation}\label{bil}e(\bu,\bw) = (\nabla \cdot \bu, \nabla \cdot \bw) + (\nabla \bu,\nabla \bw)+(\nabla \bu,\nabla \bw^T).\end{equation} 
We topologize the space $\bV$ via ~$e(\cdot,\cdot)$, which is to say that we take the norm induced by $e(\cdot,\cdot)$ as the norm on $\bV$, and, via Korn's inequality  and Poincar\'e, this is equivalent to the full $\mathbf H^1(\Omega)$ norm on $\bV$ \cite{kesavan}, as in \cite{cao,bgsw}.

In our estimates below, we utilize the notation of ~$Q_1 \lesssim Q_2$~ to indicate that there is a constant $C$ depending only on non-critical quantities such that  ~$Q_1 \le CQ_2.$ In general, throughout the paper the quantity $C$ represents a generic constant that may change from line to line. If a constant exhibits a critical dependence, this will be denoted with subscripts or in parentheses, for instance: ~$||f|| \le C_p||g||$ or ~$A=A(\Omega)$.

Finally, in this analysis {\em we assume that the principal domain $\Omega$ is of class $\mathcal C^2$} \cite{kesavan,ciarlet}, so that elliptic regularity results apply \cite{kesavan}.
\subsection{Weak Solutions}
As one can see in \cite{owc,show1,cao,frenchpaper}, for instance, there are many different notions of strong and weak solution to poroelastic systems. Our notion of solution is consistent with that provided in \cite{zenisek}, in the sense that the solution satisfies a weak space-time form of \eqref{1}--\eqref{3}.  Moreover, our weak solutions are in line with the general notion of weak solution for the time-dependent linear problem holding in the dual sense (in $L^2(0,T;V')$), as presented \cite{showmono}.

\begin{definition}[Weak solutions]\label{zerosol}
A solution to \eqref{1}--\eqref{3} with $c_0\ge 0$ is represented by the pair of functions 
$$\bu \in L^2(0,T;\bV) \ \ \text{and}\ \  p \in L^2(0,T;V),$$ 
with $\zeta = c_0p + \nabla \cdot \bu \in L^2(0,T;L^2(\Omega))\cap H^1(0,T;V')$, such that: 
\vskip.1cm
\noindent(a) the following variational forms are satisfied for any $\bw \in L^2(0,T;\bV)$, $q \in L^2(0,T;V)$:
\begin{align}\label{weakform}
\int_0^T e(\bu,\bw) dt - \int_0^T& (p,\nabla \cdot \bw)_{\Omega} dt = \int_0^T (\bF,\bw)_{\Omega} ~dt\\\label{weakform2}
\int_0^T \big(k(\zeta)\nabla p, \nabla q\big)_{\Omega} ~dt-\int_0^T&\langle\zeta_t, q\rangle_{V'\times V} ~dt = \int_0^T \langle S,q \rangle_{V'\times V} ~dt
\end{align}
(b) for every $q \in V$, the term $(\zeta(t),q)_{L^2(\Omega)}$ uniquely defines an absolutely continuous function on $[0,T]$ and the initial condition $\big(\zeta(0), q)_{L^2(\Omega)} = (d_0, q)_{L^2(\Omega)}$ is satisfied. 
\end{definition}

\begin{remark} \label{dataremark2} Alternatively to (b) above, one could assume that $d_0 \in L^2(\Omega)$ but specify that $\zeta \in C([0,T]; V')$ and $\zeta(t) \big|_{t=0} = [c_0p+\nabla \cdot \bu](t)\big|_{t=0} = d_0$ in the $H^{-1}(\Omega)$ sense. This is precisely what we will obtain through our constructions. \end{remark}

\begin{remark}  In \cite{zenisek,bgsw}, test functions are taken as space-time products, and all terms are defined in terms of spatial $L^2(\Omega)$ inner products. Our formulation is equivalent by density, as the test functions of the form $\bw(\mathbf x)f(t)$, with $\bw \in \bV$ and $f \in C_0^{\infty}(0,T)$, are dense in $L^2(0,T;\bV)$; similarly, test functions of the form $q(\mathbf x)f(t)$, with $q\in V$ and $f \in C_0^{\infty}(0,T)$, are dense in $L^2(0,T;V)$. \end{remark}

\begin{remark} We note finally that the above definition of weak solution could certainly be weakened. For instance, in the weak form of elasticity, the RHS could be replaced by $\int_0^T\langle \bF, \bw \rangle_{\bV' \times \bV}dt$ and the initial condition $d_0$ could be taken in $V'$ (being mindful of the previous approaches in \cite{frenchpaper, bgsw,cao,show1,indiana}). However, we use Definition \ref{zerosol} based on the regularity required for our construction, predominantly influenced by the presence of nonlinearity in the problem and a careful treatment of the spatial regularity of $\zeta=c_0p+\nabla \cdot \bu$.  
\end{remark}

It will be convenient in the estimates below to utilize a notation for ``source data" associated to a priori estimates obtained in the analysis of the pressure equation \eqref{weakform2}.
\begin{definition}\label{energy_data}[\textit{Notion of Source Data}]
\begin{align}
 \text{DATA}\Big|_0^T \equiv& \int_0^T\Big[||S(t)||^2_{V'}+||\bF_t(t)||^2_{\bV'}+||\bF(t)||_{\bV'}^2\Big]dt \label{data0}
\end{align}
\end{definition}

\subsection{Main Results and Comparison to Previous Literature}\label{mainresultssec}
We begin this section with the statements of the principal results, and follow them with an in-depth, technical discussion of our results in relation to the literature. It is important to note the ways in which our contributions here represent alternative proofs for similar results in the literature, and in what ways our approaches here are novel. Indeed, there is a striking amount of subtlety already present in the analysis of the associated {\em linear} Biot system.

The first auxiliary result we discuss is that of existence of weak solutions---adapted to the setting at hand and restricted to $c_0>0$---for the associated linear problem. Namely, given $z \in L^2(0,T;L^2(\Omega))$ and associated permeability $k(z)$, we want to solve:
\begin{equation}\label{abstractinf}
\begin{cases}
- \Delta \mathbf{u} - 2 \grad(\grad \cdot \mathbf{u})=-\nabla p+\bF  &~ \in L^2(0,T; \mathbf V')\\
\zeta_t-\nabla \cdot [k(z)\nabla p]=S&~ \in L^2(0,T;V')\\
\zeta=c_0p+\nabla \cdot \bu &~ \in L^2(0,T;L^2(\Omega))\\
\zeta(0)=d_0&~ \in L^2(\Omega).
\end{cases}
\end{equation}

\begin{theorem}[Linear Weak Solution]\label{linearwellp*} Let $c_0 > 0$, and assume that the permeability $k$ satisfies the hypotheses of Assumption \ref{Assumpk}. Let $d_0\in L^2(\Omega)$, $z \in L^2(0,T;L^2(\Omega))$, $\bF \in L^2(0,T;\mathbf L^2(\Omega))\cap H^1(0,T;\bV')$, and $S \in L^2(0,T;V')$. Then \eqref{abstractinf} has a  weak solution $(\bu(z),p(z), \zeta(z))$, where $\bu(z) \in L^2(0,T;\mathbf H^2(\Omega)\cap \bV)$, $p(z) \in L^2(0,T;V)$, and $\zeta(z) \in L^2(0,T;H^1(\Omega)) \cap H^1(0,T;V')$, with associated estimates:  
\begin{align}\label{placeholder1}
c_0||p||_{L^{\infty}(0,T;L^2(\Omega))}^2+\|p\|_{L^2(0,T;V)} \lesssim& ~ ||d_0||^2_{L^2(\Omega)}+ DATA|_0^T\\[.2cm] \label{placeholder2}
||\bu||^2_{L^2(0,T;\mathbf H^2(\Omega)\cap \bV)} \lesssim &~ ||p||_{L^2(0,T,V)}^2+||\bF||_{L^2(0,T;\mathbf L^2(\Omega))}^2\\[.2cm] \label{placeholder3}
\|[c_0p+\nabla \cdot \bu]_t\|_{L^2(0,T;V')} \lesssim &~  \|p\|_{L^2(0,T;V)} + \|S\|_{L^2(0,T;V')}
\end{align}
\end{theorem}
A few remarks are in order about the linear result above:
\begin{remark} We note that nothing in the above result (or its corresponding proof in Appendix B) changes if, instead of $k=k(z)$  we take $k=k(\mathbf x,t)$ to be a given $L^{\infty}(0,T;L^{\infty}(\Omega))$ - function. 
\end{remark}
\begin{remark}\label{showremark} The linear result above for $k=k(\mathbf x,t)$ was obtained earlier in \cite{indiana} (later exposited in \cite[p.116]{showmono}). Those results provide the existence of weak solutions for time-dependent implicit problems. The conditions for existence are quite general and permit $c_0\ge 0$. Moreover uniqueness results are available
with the additional hypothesis that  $k_t \in L^1(0,T;L^{\infty}(\Omega))$. For the analysis of the nonlinear problem, we cannot impose additional assumptions on the permeability, as we must apply results for $k(z(t))$ in our fixed point construction, precluding additional hypotheses on $k$.
\end{remark}

Next we present our result for the existence of weak solution to the nonlinear system:

\begin{theorem}[Nonlinear Weak Solution]\label{th:main}
Consider the nonlinear coupled system (\ref{Eq:elasticity})--(\ref{Eq:IC}) with $c_0\ge 0$, permeability function $k(\cdot)$ satisfying Assumption \ref{Assumpk}, and distributed sources $S \in L^2(0,T;V')$ and $\bF \in L^2(0,T;\mathbf L^2(\Omega)) \cap H^1(0,T;\bV')$.  For initial data~ $\zeta(0)=d_0 \in L^2(\Omega)$, there exists a weak solution  
~$\bu \in L^2(0,T;\bV) \ \ \text{and}\ \  p \in L^2(0,T;V)$ 
in the sense of Definition \ref{zerosol}.
Moreover, $c_0p \in L^{\infty}(0,T;L^2(\Omega))$ and $\bu \in L^2(0,T;\mathbf H^2(\Omega) \cap \bV)$, with
 the same inequalities \eqref{placeholder1}--\eqref{placeholder3} holding for nonlinear weak solutions.

\end{theorem}
\begin{remark} We note from the estimates above that, in the sense of the pressure equation, the solution is truly ``weak." However, since we have elliptic-parabolic coupling, with the regularity assumptions placed on $\bF$, the solution is ``strong" in the sense of the elastic displacement, since that equation holds $a.e.~\mathbf x,~a.e.~t$. 
\end{remark}

We now address uniqueness of solution through the imposition of additional regularity hypotheses. As the problem is fundamentally quasi-linear in nature, such additional regularity for uniqueness is expected. Our approach to uniqueness is rooted in multiplier estimates, which are themselves problematic for weak solutions. Hence, we need to restrict our attention to the class of weak solutions (for fixed data $\bF,S,d_0$,  permeability function, and intrinsic parameters) such that the solution can properly be used as a test function. Beyond this, additional spatial regularity will be needed to manipulate the nonlinear permeability term. 

\begin{definition}\label{Wset} Define the class of  solutions $\mathcal W_T=\mathcal W_T(\bF,S,d_0,k(\cdot))$ as consisting of weak solutions that have additional time regularity, i.e.:
\begin{equation}\label{W_T}
\mathcal W_T\equiv \big\{ (\bu,p) ~\text{is a weak solution in the sense of Definition \ref{WeakSol} on $[0,T]$ }~| ~p_t \in L^2(0,T;L^2(\Omega))\big\}.
\end{equation}
\end{definition}

\begin{theorem}[Uniqueness]\label{th:main2} Let $c_0 \ge 0$. Suppose that, in addition to Assumption \ref{Assumpk}, we have $k \in Lip(\mathbb R)$ and assume $(\bu,p) \in \mathcal W_T$. 

\begin{itemize}
\item Suppose $d_0 = \nabla \cdot \bu_0+c_0p_0 \in L^2(\Omega)$ for some $\bu_0 \in \bV$. If additionally $p \in L^2(0,T;W^{1,\infty}(\Omega))$, then the solution $(\bu,p)$ is unique in the class $\mathcal W_T$. 
\item If $c_0>0$ and $p \in L^2(0,T;W^{1,\infty}(\Omega))$, then  $(\bu,p)$ is unique in the class $\mathcal W_T$. 

\end{itemize}

In each of the above cases, if one solution in $\mathcal W_T$ has the appropriate additional spatial regularity for $p$, then all solutions in $\mathcal W_T$ are equal. 
\end{theorem}
\begin{remark}
By the standard Sobolev embeddings \cite{kesavan}, it is sufficient for the theorem above to have $p \in L^2(0,T;H^{2.5+\delta}(\Omega))$ for any $\delta>0$.
\end{remark}
\begin{remark} We note the two above cases sacrifice one hypothesis at the cost of another. In the second bullet point, some compressibility needs to be assumed, but no additional structure of the data $d_0$ need be assumed. In the first bullet, we can take $c_0=0$ but require information about $\bu(t=0)$ to be independently specified. We point out here that these conditions are an improvement of those in \cite{cao}; in that reference they require $c_0>0$ and also impose a smallness condition on $\nabla p$ in ${L^{\infty}\big((0,T)\times \Omega\big)}$ in terms of the intrinsic parameters. We also mention that, to the best of our knowledge, no previous work actually constructs strong solutions. \end{remark}

\noindent {\bf Challenges and Relation to Previous Literature:}  The main mathematical challenges in this problem are represented by (i) the implicit, degenerate evolution present in the system, as well as (ii) the nonlinear coupling (with no evident monotone structure) in the permeability---it being a function of fluid content. There is substantial mathematical literature focused on well-posedness analysis for  linear poroelastic systems, where the permeability tensor is assumed to be constant. The key references in the linear setting are   \cite{frenchpaper,zenisek, indiana, owc, show1}. 

A foundational reference for all of the cited mathematical Biot studies is \cite{zenisek}. This paper provides a construction of solutions in the {\it 2-D linear case} using Rothe's method (full temporal and spatial discretization), with the analysis based on a priori estimates. The analysis is done on the entire $(\bu,p)$ system, and, as such, requires the specification of initial displacement $\bu(0) \in V^2$ and initial pressure $p(0) \in L^2(\Omega)$. In contrast, the seminal work  in \cite{frenchpaper, showmono,show1} reduces the full linear Biot system to an implicit, degenerate evolution. This allows---again in the linear case---a modified semigroup theory to obtain both weak and strong solutions, and uniqueness is addressed. 
For the weak solutions (what \cite[Section 6]{show1} calls the ``holomorphic case") only specification of the initial fluid content $\zeta(0) \in H^{-1}$ is needed. The implicit semigroup approach works in quotient and seminormed spaces to reduce the implicit problem to that of a regular explicit Banach-valued ODE \cite[Chapter IV.6]{showmono}. Estimates in this setting are obtained in the dual domain for associated resolvents. As such, the approach is not immediately generalizable to time-dependent and/or quasilinear cases. 

We note that in \cite[Chapter III.3]{showmono}, a nice formulation for weak solutions to the linear, time-dependent problem is presented based on a generalized version of Lax-Milgram due to Lions (an approach that originally appeared in \cite{indiana}). As pointed out above,  that theory is applicable to the linear problem at hand with comparable results to our Theorem \ref{linearwellp*}. Lastly, with respect to the linear analysis, the more recent \cite{owc} provides a Galerkin-based construction of solutions for the full Biot problem, making use of an explicit solver for the embedded Stokes-type problem in the dynamics. There, solutions are clearly constructed without temporal discretization, but strong assumptions  are made on the data in order to obtain good a priori estimates. 

In the authors' previous work \cite{bgsw},  the nonlinear problem presented here is addressed with $c_0=0$ and allows for the possibility of visco-elastic effects in the Biot structure. Additionally, motivated by physical considerations, a configuration  with mixed boundary conditions on a Lipschitz domain, and non-zero boundary sources, is considered. That work is based on full spatio-temporal discretization (adapting the linear argument in \cite{zenisek}). As such we make use of the stronger assumption on initial data in order to obtain good estimates at the temporally discretized level and carefully pass with the limit, invoking compactness in the fluid content derived from a small amount of elliptic regularity. The multipliers approach works there, albeit in the discrete setting, with two subsequent limit passages required. The accompanying estimates are less natural however, and the fully discretized nature is neither optimal nor natural for modern numerical analysis of the nonlinear problem. 

In comparison, our goal here is to provide a theory of solutions  for the nonlinear poroelastic coupling in (\ref{Eq:elasticity})--(\ref{Eq:IC}). As mentioned before, we consider a similar model and set of assumptions as the ones used in  \cite{cao}. However, we permit the case of fluid-solid mixtures which may have incompressible constituents ($c_0=0$), with applications to biological tissues. This degeneracy is rather benign at the linear level, but presents subtle challenges for the analysis here, owing to the fact that the key operator $B$ (as seen in Section \ref{Bdef})  is not invertible on $L^2(\Omega)$, but $c_0\mathbf I+B$ is. Indeed, we use critically the presence of $c_0>0$ to construct solutions, and then produce the solution for $c_0=0$ via a singular limit approach as $c_0 \searrow 0$. It is also worthwhile to note that, in line with biological applications, we allow the permeability to depend on the {\em full} fluid content, i.e., $k(\zeta)$ for $\zeta=c_0p+\alpha \nabla \cdot \bu$ (also as in \cite{bgsw}). In \cite{cao}, the construction critically requires $c_0>0$ \cite[p.1259(line 6), p.1260(line -9)]{cao} but there, {\em the permeability depends only on dilation}, i.e., $k(\nabla \cdot \bu)$ \cite[p.1254]{cao}; this distinction is mathematically non-trivial.

We present here what we believe to be the most direct and illustrative approach for existence of weak solutions. Our approach does not involve the discretization of the balance equations in both time and space \cite{cao, bgsw}. We believe this is beneficial for future considerations, as full discretization is cumbersome for a sought-after construction of smooth solutions, and our semi-discretized approach is perhaps more amenable to numerical treatment. The work in \cite{show1} focuses on constant permeability $k$ (which renders a linear coupling in the system) and develops a semigroup theory for implicit evolution equations for both strong and weak solutions. The approach is generalized for the case of nonlinear permeability function dependent on pressure which preserves a monotone structure in \cite{show2}. We note that the strategy developed in \cite{show1} can not be directly applied here, as the model at hand does not exhibit such monotonicity properties. Rather, we build linear time dependent solutions and carefully construct a functional correspondence that leads to a fixed point. In constructing weak solutions via  estimates in a fixed point argument, we hope to have provided a framework for the future construction of smooth solutions, which should be unique, according to the criterion given here.

\section{Fundamental Operators and Translation of Momentum Source}\label{opsprops}
In this section we introduce the principal operators that are used in the proofs of the main theorems, along with their properties. We  
follow the abstract framework provided in \cite{show1}. 
In the last part of the section we provide formal ``translations" of the linear and nonlinear problems that allow us to consider the problem with null distributed force in the balance of linear momentum, upon translating the initial data and the pressure source $S$. 

In what follows, it will be necessary to invoke the gradient and divergence operators \cite{show1,temam}, and their dual relationship with respect to $\bV = \mathbf H_0^1(\Omega)$ and $\bV'=\mathbf H^{-1}(\Omega)$. Namely, from \cite{temam}, the standard gradient
$\nabla : L^2 \to \mathbf V'$ ~has  dual operator ~$-[\nabla \cdot] : \mathbf V \to L^2(\Omega)$, with both acting boundedly in those settings. More can be said;
utilizing the standard (abuse of) notation for the quotient space ${L^2_0(\Omega) =}L^2(\Omega) /\mathbb R \equiv \{ f \in L^2(\Omega) ~:~ \int_{\Omega} f =0\},$ we note that the divergence~
$[\nabla \cdot] : \mathbf V \to L^2_0(\Omega)$ is invertible and $\nabla : L^2_0(\Omega) \to \mathbf V'$ is an isomorphism.

\subsection{Elasticity Operator $\cE$}\label{elasticop}
In general, the elasticity operator associated to isotropic homogenous media is given by $$-(\lambda_e+\mu_e)\nabla(\nabla \cdot \bu)-\mu_e \Delta \bu=-\nabla\cdot[2\mu_e \varepsilon(\bu)+\lambda_e (\nabla \cdot \bu)\mathbf I].$$ Since we have taken $\mu_e=\lambda_e=1$ here, we consider an operator $\cE(\bu)$, whose action in distribution is given by:
$$\cE_0(\bu) = -\nabla \cdot [2\varepsilon(\bu)+\nabla\cdot \bu]=-2\nabla(\nabla \cdot \bu)-\Delta \bu.$$ This differential action is naturally associated to the symmetric bilinear form $e(\cdot,\cdot): \bV \times \bV \to \mathbb R$ given in \eqref{bil}. In the standard way, for $ \bu \in \bV$, we can define $\mathcal E(\bu) \in \mathbf V'$ by  $\langle \mathcal E(\bu), \cdot \rangle_{\mathbf V'\times \mathbf V}$. By restricting the action of the bilinear form, we can identify an unbounded operator $\cE: \mathbf L^2(\Omega) \to \mathbf L^2(\Omega)$ that encodes the homogeneous Dirichlet boundary conditions.  This is to say, $\mathcal E$ is the operator with domain 
$$\mathcal D(\mathcal E) \equiv \{ \bu \in \bV~:~\mathcal E_0(\bu) \in \mathbf L^2(\Omega)\}.$$
Indeed, $\cE$ as above is a $\mathbf V \to \bV'$ isomorphism \cite{show1},  and $\mathcal E: \mathbf L^2(\Omega) \to \mathbf L^2(\Omega) $ is positive, self-adjoint, and an isomorphism from $\mathcal D(\cE) \to \mathbf L^2(\Omega)$ (the latter invokes elliptic regularity and is stated precisely below).

In the analysis of the momentum equation, we consider a given a $p \in L^2(\Omega)$ (and thus $\nabla p \in \bV'$ \cite{temam,show1}) and produce a corresponding $\bu \in \bV$ which satisfies the stationary elasticity equation, which we will frequently write as 
\begin{equation}\label{convenient}\cE(\bu)=-\nabla p+\bF \in \mathbf V'.\end{equation}

This leads directly to the following lemma:
\begin{lemma}\label{elasticity} Given $\mathbf G \in \mathbf V'$,  we can consider the elasticity problem 
\begin{equation}\begin{cases}\label{ptodmap}
\cE(\bu) = \mathbf G & \in \bV' \\
\bu = 0 &~ \text{on}~\Gamma.
\end{cases}
\end{equation}
This problem is well-posed in the standard weak sense \cite{kesavan,ciarlet}, with a solution $\bu \in \bV$ and  stability estimate
$$||\bu||_{\mathbf V} \le C_w || \mathbf G||_{\mathbf V'},~~\forall \mathbf u \in \mathbf V.$$

Moreover, as $\Omega$ is of class $\mathcal C^2$, classical elliptic regularity applies \cite{ciarlet,temam}. Hence, if $\mathbf G \in \mathbf L^2(\Omega)$, then we have that $\mathbf u \in \mathbf H^2(\Omega)\cap \bV,$ and 
$$||\bu ||_{\mathbf H^2(\Omega)} \le C_r||\mathbf G||_{\mathbf L^2(\Omega)}.$$
 \end{lemma}

Unlike \cite{bgsw}, we are working  with a smooth boundary, composed of a single Dirichlet component upon which both pressure $p$ and displacement $\bu$ are zero.  Thus classic elliptic theory can be used  for displacement $\bu$ when $p \in V$ and $\bF \in \mathbf L^2(\Omega)$. When $\mathbf F = \mathbf 0$ in \eqref{convenient}, we have
\begin{equation}\label{ptod} p \in V ~\implies\nabla p \in L^2(\Omega)~\implies~\cE^{-1}(-\nabla p) = \bu \in \mathbf H^2(\Omega)\cap \bV~\implies~\nabla \cdot \bu \in \mathbf H^1(\Omega).\end{equation} Such regularity was not available in \cite{bgsw}, where a more complex, physically-motivated boundary configuration was considered.

\subsection{Diffusion Operator $A_{z(t)}$}

For a smooth $z \in C^1([0,T] \times \Omega)$, we define the linear operator $A_{z(t)}: V \to V'$ by 
\begin{equation}\label{Aaction} A_{z(t)} p = -\nabla \cdot[ k (z(t)) \nabla p],~~\forall p \in  V,\end{equation}
where $k(z(t))$ is interpreted  as a Nemitskii operator for  the given function $z(\mathbf x,t)$ as in \ref{Assumpk}. 
\begin{remark} In practice we will consider this operator through its bilinear form (defined below) when $z \in L^2(0,T;L^2(\Omega))$, considered a.e. $t$.\end{remark}

If we assume that $z \in L^2(0,T; H^1(\Omega))$, then we have an unbounded operator $A_{z(t)}: L^2(\Omega) \to L^2(\Omega)$ with domain $\mathcal D(A_{z(t)}) = H^2(\Omega) \cap V$ and action given by \eqref{Aaction} with associated  bilinear form
\begin{equation}
A[p,q; z(t)]= (k(z(t)\nabla p,\nabla q),~~\forall~p,q \in V,~a.e.~t \in [0,T].
\end{equation}
As noted above, the bilinear form associated to the weak form of $A_{z(t)}$ requires only that $z \in L^2(0,T;L^2(\Omega))$, and can be obtained via density.
When $k\equiv ~const$, $A_{z(t)}$ is just a multiple of the standard Dirichlet Laplacian. In the setting at hand, for a.e. $t\in[0,T]$, (i) $A_{z(t)}$ is a maximal monotone operator, (ii) the bilinear form $A[\cdot, \cdot; z(t)]$ is continuous and coercive on $V$, and (iii) $A_{z(t)}$ is positive and self-adjoint as an unbounded operator on $L^2(\Omega)$.

\subsection{Pressure to Dilation Map}\label{Bdef}
The pressure to dilation map was introduced in the setting of Biot problems in \cite{frenchpaper}, and developed and used extensively in \cite{showmono,show1}. It allows one to reduce the $(\bu,p)$ system in \eqref{1}--\eqref{3} to an implicit evolution problem, such as those studied extensively in \cite{showmono,show1}. This operator is a useful, descriptive tool in the construction of approximate solutions, and the subsequent analyses.

Consider the map $B: L^2(\Omega) \to L^2(\Omega)$, defined through the gradient and divergence operators (as above) by:  
\begin{equation} \label{Bdef*} Bp =-\nabla \cdot \mathcal E^{-1}(\nabla p),\end{equation} motivated by the problem above in \eqref{ptodmap}. Indeed, we have that $p \in H^s(\Omega) \implies \nabla p \in \mathbf H^{s-1}(\Omega)$ with $p \mapsto \nabla p$ continuous in that setting \cite{temam,kesavan}. In the specific case when $p \in L^2(\Omega)$,~~ $\nabla p \in \mathbf H^{-1}(\Omega)=\mathbf V'$. Invoking the properties of the elliptic operator $\cE$, we see that indeed $B \in \mathscr L(L^2(\Omega))$. 
Similarly, the action of $B$ extends readily to $V$. Considering $B$ as above, if $p \in V$, then as above in \eqref{ptod}, $Bp \in H^1(\Omega)$
Similarly, we obtain immediately that $B \in \mathscr L(H^s(\Omega)\cap \bV, H^s(\Omega))$ for $s \ge 1$ via the elliptic regularity associated to $\cE$. 
Note that when $\bF\equiv 0$ in the elasticity equation, as in \eqref{ptod}, we have $Bp = \nabla \cdot \bu$.
\begin{remark} Consider $p \in V$, and, as above $Bp = \nabla \cdot \bu \in H^1(\Omega)$.
Although it is not clear that $\nabla \cdot \bu \in V$, we do know that $Bp=\nabla \cdot \bu \in \text{div}[\bV]$, which does carry additional information, in particular that $Bp \in H^1(\Omega) \cap L^2(\Omega)/\mathbb R$. Moreover, $B \in \mathscr L(L^2_0(\Omega))$.\end{remark}

Consolidating the discussions above we have:
\begin{lemma}\label{elipreg} Given $p \in V$ and $\bF\in \mathbf L^2(\Omega)$, the corresponding solver $\cE^{-1}(-\nabla p+\bF)\in \mathbf H^2(\Omega) \cap \bV$ with associated continuity bound. When $\bF\equiv 0$ and $p \in V$, we have $Bp = \nabla\cdot \bu \in H^{1}(\Omega)$ for $\cE(\bu)=-\nabla p$. From this we obtain  that 
$$B: H_0^1(\Omega) \to H^1(\Omega),~~{\text{continuously}}.$$ 
\end{lemma}

We note some kernel and range properties of the $B$ operator (closely following \cite{show1} and utilizing the properties of $\nabla$ and $\text{divergence}$ \cite{temam}). 
 \begin{lemma}\label{Binvert}
Considered as a mapping on $H_0^1(\Omega)$, $B$ is injective. Considered as a mapping on $L^2(\Omega)$, $ker(B)=\{\text{constants}\}$, and hence $B$ is injective on $L^2(\Omega)/\mathbb R$.
 
With respect to ranges, we have the following:
$$B(L^2(\Omega)) \subseteq L^2(\Omega)/\mathbb R,~~B(H_0^1(\Omega)) \subseteq H^1(\Omega)/\mathbb R.$$
\end{lemma}

Finally, we have that $B$ is a self-adjoint, monotone operator when considered on $L^2(\Omega)$ \cite{show1,bgsw}.
\begin{lemma}\label{posop} Considering  $B \in \mathscr L(L^2(\Omega))$, it is a non-negative, self-adjoint operator.
\end{lemma}
By the standard construction \cite{pazy,ciarletbook}, the self-adjoint operator $B^{1/2} \in \mathscr L(L^2(\Omega))$ is obtained with characterizing property
$$(Bp,q)_{L^2(\Omega)}=(B^{1/2}p,B^{1/2}q)_{L^2(\Omega)} = (p,Bq)_{L^2(\Omega)},~~\forall~p,q \in L^2(\Omega).$$
A central issue in the analysis here is that $B \in \mathscr L(L^2(\Omega))$ need not be coercive in that setting. However, in the case where we consider compressible effects,  the operator $c_0\mathbf I+B : L^2(\Omega) \to L^2(\Omega)$ is coercive. We will use this critically and repeatedly below.

\begin{corollary}\label{Binvert*} Let $c_0>0$. Then the operator $c_0\mathbf I +B~:~L^2(\Omega) \to L^2(\Omega)$ is an isomorphism. \end{corollary}
\begin{proof}
By positivity, we note that for $p \in L^2(\Omega)$ 
$$([c_0\mathbf I +B] p,p)_{L^2(\Omega)} = c_0||p||^2+(Bp,p) \ge c_0||p||^2.$$ Hence, the operator $c_0\mathbf I+B$ is coercive on $L^2(\Omega)$. Since $B \in \mathscr L(L^2(\Omega))$, surjectivity follows immediately from classical Lax-Milgram applied to the form
$$\beta(p,q) = \big([c_0\mathbf I+B]p,q\big)_{L^2(\Omega)}.$$
 Which is to say that $$ \beta(p,q)= (f, q)_{\Omega},~~\forall q \in L^2(\Omega)$$ is uniquely solvable with associated stability bound. 
\end{proof}

We conclude this section with some additional remarks about the $B$ operator that arise critically in the context of previous approaches to the problem at hand.
\begin{itemize}
\item Since $B$ implicitly invokes an elliptic solver, its behaviors on $V'$ or any $H^{-s}(\Omega)$---regularity and continuity properties, kernel and range---are {\bf not} in the realm of standard elliptic theory.
\item It is not clear that $B \in \mathscr L (H^{-1}(\Omega))$, or that $B(H_0^1(\Omega)) \subset H_0^1(\Omega)$. \end{itemize}

\subsection{Translation to Eliminate Momentum Source $\bF$}\label{translation}
In this subsection we provide formal ``translations" to the linear and nonlinear problems that allow us to consider the problem with $\bF \equiv 0$. In latter sections, to simplify the analysis, we will operate on the translated problem. After obtaining the principal result in those sections, we will refer to this section and translate back to obtain (linear and nonlinear) results for the original system.

We first note that it is sufficient to solve the {\em linear problem}---where $k=k(z)$ for $z \in L^2(0,T;L^2(\Omega))$ a given function---with $\bF \equiv 0$ by a standard translation. Consider 
\begin{equation}\label{abstractform}
\begin{cases}
\cE(\bu)=-\nabla p  &~ \in L^2(0,T; \mathbf V')\\
\zeta_t+A_{z(\cdot)} p=S&~ \in L^2(0,T;V')\\
\zeta = c_0p+\nabla \cdot \bu &~\in L^2(0,T; L^2(\Omega))\\
\zeta(0)=d_0&~ \in L^2(\Omega)
\end{cases}
\end{equation}
Indeed, as the elasticity equation is elliptic and $\bF \in L^2(0,T;\mathbf L^2(\Omega))$, for a.e. ~$t \in [0,T]$ we can simply write \begin{equation}\label{ellipticsolver}\bu_{\bF}(t) = \cE^{-1}(\mathbf F(t)) \in \mathbf H^2(\Omega) \cap \bV.\end{equation} Additionally, with  the regularity hypotheses of our main theorem  $\bF \in L^2(0,T;\mathbf L^2(\Omega)) \cap H^1(0,T;\bV')$, we have that $\bu_{\bF} \in L^2(0,T; \mathbf H^2(\Omega) \cap \bV)\cap H^1(0,T; \bV)$. Then, considering the variable $\bw = \bu-\bu_{\bF}$, we note that $\bu$ solves \eqref{abstractform} if and only if $\bw$ solves
\begin{equation}\label{abstractform*}
\begin{cases}
\cE(\bw)=-\nabla p  &~ \in L^2(0,T; \mathbf V')\\
c_0 p_t + \nabla \cdot \bw_t+A_{z(\cdot)} p=S+\nabla \cdot \bu_{\bF,t}&~ \in L^2(0,T;V')\\
c_0 p(0) + \nabla\cdot\bw(0)=d_0-\nabla \cdot \bu_{\bF}(0)&~ \in L^2(\Omega).
\end{cases}
\end{equation}
Hence, by re-scaling $S \in L^2(0,T;V')$ and $d_0= \zeta(0) \in L^2(\Omega)$, we obtain an equivalent linear problem for a given $z$ with $\bF \equiv 0$.

In the case of the nonlinear problem, where $A_{z(t)}=A_{\zeta(t)}$ for $\zeta = c_0p+\nabla \cdot \bu$, an additional step is needed. We note that if $\bw= \bu-\bu_{\bF}$ as above, the fluid content has expression $$\zeta=c_0p+\nabla\cdot \bu=c_0p+\nabla \cdot \bw+\nabla \cdot \bu_{\bF}.$$ Since $\nabla \cdot \bu_{\bF} \in L^2(0,T; H^1(\Omega)) \cap H^1(0,T; L^2(\Omega))$, we have that $\nabla \cdot \bu_{\bF} \in C([0,T];L^2(\Omega))$ \cite{evans,showmono}. 
Hence, for any $k(\cdot)$ as in \eqref{Assumpk}, we introduce the function $k_{\bF}(\cdot) \in C(\mathbb R)$ representing the $\bu_{\bF}$-translate of $k(\cdot)$, namely \begin{equation}\label{ktrans} k_{\mathbf F}(\cdot)= k(\cdot +\nabla \cdot \bu_{\bF}).\end{equation}
Since $k(\cdot)$ satisfies \ref{Assumpk}, we obtain immediately that  $k_{\bF}(\cdot)$ satisfies the assumption as well.
Then, from the (abstract) strong form of the original problem,
\begin{equation}\label{abstractform**}
\begin{cases}
\cE(\bu)=-\nabla p+\bF  &~ \in L^2(0,T; \mathbf V')\\
\zeta_t-\nabla \cdot [k(\zeta(\cdot))\nabla p]=S&~ \in L^2(0,T;V')\\
\zeta=c_0p+\nabla \cdot \bu &~ \in L^2(0,T;L^2(\Omega))\\
c_0 p(0) + \nabla\cdot\bu(0)=d_0&~ \in L^2(\Omega),
\end{cases}
\end{equation}
we can write the system for $\bw = \bu-\bu_{\bF}$ as follows
\begin{equation}\label{abstractform***}
\begin{cases}
\cE(\bw)=-\nabla p  &~ \in L^2(0,T; \mathbf V')\\
[c_0 p + \nabla \cdot \bw]_t-\nabla \cdot [k_{\bF}(c_0p+\nabla \cdot \bw )\nabla p]=S+\nabla \cdot \bu_{\bF,t}&~ \in L^2(0,T;V')\\
c_0 p(0) + \nabla\cdot\bw(0)=d_0-\nabla \cdot \bu_{\bF}(0)&~ \in L^2(\Omega).
\end{cases}
\end{equation}
As in the linear case, for a fixed $\bF$ and $k$, by re-scaling $S$ and $d_0$ accordingly, and re-labeling $k \mapsto k_{\bF}$, we again obtain an equivalent problem with $\bF=0$.

\section{Existence of Solutions for the Linear Problem}\label{linearize}
In this section we recapitulate the relevant {\em linear} Theorem \ref{linearwellp*} to be invoked in our fixed point construction. As we noted in the Introduction and Remark \ref{showremark}, the abstract theory of time-dependent linear evolutions from \cite{indiana,showmono} produces an equivalent result. For self-containedness, and to tailor the analysis to the specific estimates utilized in later sections, we provide a brief discussion of the setup here, and also a traditional Galerkin construction of solutions in Appendix B.   \\

Recall that we explicitly assume that $c_0>0$. In line with the translation introduced in Section \ref{translation}, we consider the problem with $\bF \equiv 0$. Moreover, we retain the names for the datum $d_0 \in L^2(\Omega)$ and source $S \in L^2(0,T;V')$ (after updating them, as discussed in the previous section). Thus we consider the linear problem \eqref{abstractform} for a given $z \in L^2(0,T;L^2(\Omega))$.
Interpreting it weakly through the associated bilinear forms $e(\cdot,\cdot)$ and $A[\cdot,\cdot;z(t)],$ and taking the time derivative distributionally in $\mathscr D'(0,T)$ yields an equivalent formulation to our Definition \ref{zerosol} (see \cite{showmono}). Namely, we seek $\bu(z) \in L^2(0,T;\bV)$, $p(z) \in L^2(0,T;V)$ that solve \eqref{abstractform} weakly.

Using the pressure to dilation operator introduced in Section 3, we equivalently reformulate \ref{abstractform} as was done in \cite{show1,cao}  as the (implicit) initial  boundary value problem
\begin{equation}\label{weakpforgiventildep} 
\begin{cases}
[(c_0\mathbf I + B)p]_t -\nabla \cdot [k(z)\nabla p] = S, &  \Omega \times (0,T)\\
p = 0, & \Gamma \times (0,T)\\
(c_0\mathbf I+ B)p(0)= d_0, & \Omega
\end{cases}
\end{equation}

Recall that for $z \in L^2(0,T;L^2(\Omega))$, we defined the bilinear form 
\begin{equation}\label{Bil_form_A}
A[\cdot, \cdot;z(t)] : V \times V \to \reals \ \ \text{by}
\ \ A[p, q;z(t)] = (k(z(t))\nabla p,\nabla q)_{\Omega}.
\end{equation}

\medskip

Let us clearly state a weak formulation for (\ref{weakpforgiventildep}) above.
In the remaining part of this section, the notation $f'$ will denote differentiation in time, and recall the angle brackets  represent the duality pairing between $V'$ and $V$, i.e., $\langle \phi, g\rangle = \phi(g)=\langle\phi, g\rangle_{V'\times V}$. 

\begin{definition}\label{WeakSol} Given $z \in L^2(0,T;L^2(\Omega))$,
we say that $p \in L^2(0,T; V)$ with  $(c_0\mathbf I  + B)p \in L^2(0,T; H^1(\Om))$ and $[(c_0  + B)p]' \in  L^2(0,T; V')$ is a weak solution for (\ref{weakpforgiventildep}) provided that
\begin{enumerate}
\item For every $q \in L^2(0,T; V)$, 
\begin{equation}\label{varform1}
\int_0^T \big\langle[(c_0\mathbf I  + B)p]'(t), q(t)\big\rangle \ dt + \int_0^T A[p(t), q(t);z(t)]\ dt = \int_0^T \langle S(t), q(t)\rangle dt
\end{equation}
\item $\big[(c_0\mathbf I  + B)p\big](0) = d_0$ in the sense of $V'$. 
\end{enumerate}
\end{definition}

Note that since $(c_0 \mathbf I + B)p \in L^2(0,T; H^1(\Om))$ and $[(c_0 \mathbf I  + B)p]' \in  L^2(0,T; V')$, we have that $(c_0 \mathbf I  + B)p \in C([0,T]; V')$ and thus the initial condition makes sense in $V'$.
\begin{remark} With regard to the initial condition, although we only identify the initial condition in the sense of $V'$ (as is consistent with the general weak formulation for implicit equations) we also need $d_0 \in L^2(\Omega)$ for the construction at hand. Secondly, since we do not know that $[c_0\mathbf I+B]p \in V$, we cannot use the standard result \cite{evans} to obtain that $[c_0\mathbf I+B]p\in C([0,T];L^2(\Omega))$. \end{remark}

\medskip

\begin{remark}\label{equivvarforms}
Sometimes it is convenient to work with the following equivalent variational formulation for (\ref{varform1}):
\begin{equation}\label{varform1'}
\big\langle[(c_0 \mathbf I + B)p]', q\big\rangle + A[p, q;z(\cdot)] = \langle S, q\rangle,\ \text{for each}\ q \in V \ \ \text{and a.e. time}\ \ t \in [0,T].
\end{equation} 
We note that the real-valued function $t \mapsto \langle [(c_0\mathbf I  + B)p]', q\big\rangle $ belongs to $L^1_{loc}(0,T) \subset \mathcal{D}'(0,T)$,  and Bochner's theorem yields
$$\big\langle [(c_0 \mathbf I  + B)p]', q\big\rangle= \frac{d}{dt} ([c_0\mathbf I  + B]p(t), q)_{\Omega}\ \ \text{in} \ \ \mathcal{D}'(0,T).$$
Therefore \eqref{varform1'} can be  simply written in the form
$$ \frac{d}{dt} \big([c_0 \mathbf I  + B]p(t), q\big)_{\Omega} + A[p, q;z(t)] = \langle S(t), q\rangle,~\text{ in }~\mathcal{D}'(0,T), ~\text{ for all }~ q \in V. $$
\end{remark}

We now assert the existence of a weak solution as in Definition \ref{WeakSol} to the reduced problem described above in \eqref{weakpforgiventildep}. We state this as a lemma, as it will be used in the proof of Theorem \ref{th:main}; its proof is given in Appendix B.

\begin{lemma}\label{F_welldefined}
Let $z \in L^2(0,T;L^2(\Omega))$, $S \in L^2(0,T;V')$ and $d_0 \in  L^2(\Om)$. Then \eqref{weakpforgiventildep} has  weak solution, according to Definition \ref{WeakSol}.
\end{lemma}

\begin{proposition}\label{estsfollow} The weak solution constructed in Lemma \ref{F_welldefined} satisfies for a.e. $t \in [0,T]$ the estimates
\begin{align}\label{preest1}
c_0||p(t)||_{L^2(\Omega)}^2+||B^{1/2}p(t)||^2_{L^2(\Omega)}+\|p\|_{L^2(0,T;V)} \lesssim &~ ||d_0||^2_{L^2(\Omega)}+ \|S\|^2_{L^2(0,T;V')}  \\[.2cm] \label{preest2}
\|[(c_0 \mathbf I + B)p]'\|_{L^2(0,T;V')} \lesssim &~  \|p\|_{L^2(0,T;V)} + \|S\|_{L^2(0,T;V')} \\
\|(c_0\mathbf I  + B)p\|_{L^2(0,T;H^1(\Om))} \lesssim &~ \|p\|_{L^2(0,T;V)}. \label{preest3}
\end{align}
\end{proposition}

The proof of Theorem \ref{linearwellp*} follows immediately, given the solution to \eqref{weakpforgiventildep}.
\begin{proof}[Proof of Theorem \ref{linearwellp*}] Given the lemma above, for a given $z \in L^2(0,T;L^2(\Omega))$, we have obtained the functions
$$p(z) \in L^2(0,T;V),~~\zeta(z) = c_0p(z)+Bp(z) \in L^2(0,T;H^1(\Omega)),$$
where we have denoted the dependence of the solution on the given function $z$. Since $\nabla p(z) \in L^2(\Omega)$ a.e. $t$, we can invoke the elasticity isomorphism as in Section \ref{elasticop} to obtain
$$\bu(z) = \cE^{-1}(-\nabla [p(z))]) \in \mathbf H^2(\Omega) \cap \bV,~~\text{a.e.}~t.$$
By the injectivity of $B$ on $V=H_0^1(\Omega)$, we can  identify $B[p(z)]=\nabla \cdot [\bu(z)]$ (as in \eqref{abstractform}). Hence, we obtain a solution to \eqref{abstractform}.

Finally, with the hypotheses on $\bF$ and $d_0$, we can translate back to the original case as in \eqref{abstractform*} to immediately obtain a weak solution for 
\begin{equation}
\begin{cases}
\cE(\bu)=-\nabla p+\bF  &~ \in L^2(0,T; \mathbf V')\\
\zeta_t-\nabla \cdot [k(z)\nabla p]=S&~ \in L^2(0,T;V')\\
\zeta=c_0p+\nabla \cdot \bu &~ \in L^2(0,T;L^2(\Omega))\\
c_0 p(0) + \nabla\cdot\bu(0)=d_0&~ \in L^2(\Omega).
\end{cases}
\end{equation}
\end{proof}

\section{Nonlinear Problem - Existence of Solutions}
This section contains the proof of Theorem \ref{th:main}. We divide the proof into two parts. First, we focus on the case of compressible constituents, i.e., $c_0 > 0$. Our strategy in this scenario is to show that the map $z \mapsto \zeta$ provided by Theorem \ref{linearwellp*} has a fixed point. In part two of the proof, we obtain existence of solutions for the case of incompressible mixture constituents using a limiting process $c_0 \searrow 0$. 

\subsection{Proof of Theorem \ref{th:main} for $c_0 > 0$}

We begin by considering the general translated problem with $\mathbf F \equiv 0$. By Theorem \ref{linearwellp*}, given $z \in L^2(0,T;L^2(\Omega))$, the  problem \eqref{abstractform} (with associated regularity of data) has a {\em weak solution}, written as $(\bu(z),\zeta(z),p(z))$, satisfying the  estimates in \eqref{preest1}--\eqref{preest3}.
We note that since the solution to the linear problem provided by Theorem \ref{linearwellp*} is not necessarily shown to be unique, we must allow the possibility that the solution mapping is multi-valued (in the sense of the Appendix A). Thus we consider the reduced problem
\begin{equation}\label{refthisone}
\begin{cases}
\zeta_t- \nabla \cdot [k(z)\nabla p]=S&~ \in L^2(0,T;V')\\
\zeta = c_0p+Bp &~\in L^2(0,T; L^2(\Omega))\\
\zeta(0)=d_0&~ \in L^2(\Omega),
\end{cases}
\end{equation}
and define 
a correspondence between the given {\em permeability argument} $z \in L^2(0,T;L^2(\Omega))$ and the resulting fluid contents $\zeta$ taken from weak solutions (in the sense of Definition \ref{WeakSol}) corresponding satisfying the a priori estimates in \eqref{preest1}--\eqref{preest3}.  

\begin{definition}\label{map} For fixed ``data"~ $d_0$ and  $S$ as above, and given $z \in L^2(0,T;L^2(\Omega)$, we define the correspondence 
$$\mathscr F: L^2(0,T;L^2(\Omega)) \twoheadrightarrow L^2(0,T;L^2(\Omega)),$$  by
$$\mathscr F(z) = \big\{\zeta = [c_0\mathbf I+B]p ~:~(p, \zeta) \ \text{is a weak solution of \eqref{refthisone} that satisfies \eqref{preest1}--\eqref{preest3}} \big\}.$$
Clearly, using Lemma \ref{F_welldefined} and Corollary \ref{estsfollow}, we have that the set $\mathscr F(z)  \neq \emptyset$. Moreover, the fact that the range of the correspondence $\mathscr R(\mathscr F)\subseteq L^2(0,T;L^2(\Omega))$ is immediate, since all the elements in the set $p(z)$ belong to $L^2(0,T; L^2(\Omega))$ trivially, and hence by the boundedness of $B \in \mathscr L(L^2(\Omega))$ we have that 
$$[c_0\mathbf I+B]p \in L^2(0,T;L^2(\Omega)), \ \forall p \in p(z).$$
\end{definition}
Note here that by the definition of the correspondence, the satisfaction of the initial condition is included in the definition of $\mathscr F$. Also, note that  passing between a $\zeta(z)$ and a $p(z)$ simply uses the invertibility of $[c_0\mathbf I+B]$ on $L^2(0,T;L^2(\Omega))$, via Corollary \ref{Binvert*}.\\

We will use the Bohnenblust-Karlin Fixed Point Theorem for correspondences \cite{guide} to obtain the existence of (at least) one fixed point for $\mathscr F$. The statement of the theorem, along with the relevant background definitions can be found in the Appendix A. 
We have the following theorem:
\begin{theorem}\label{th:supporting}
The correspondence $\mathscr F: L^2(0,T;L^2(\Omega)) \twoheadrightarrow L^2(0,T;L^2(\Omega))$ defined above has a fixed point. The set of fixed points of $\mathscr F$ is compact in $L^2(0,T;L^2(\Omega))$.
\end{theorem}

\begin{proof}[Proof of Theorem \ref{th:supporting}] Let $d_0 \in L^2(\Omega)$ and $S \in L^2(0,T; V')$ be given. 
We consider the correspondence  $\mathscr F: L^2(0,T;L^2(\Omega)) \twoheadrightarrow L^2(0,T;L^2(\Omega))$ defined above in Definition \ref{map}. 
By the construction of {\em linear solutions} as given in Lemma \ref{F_welldefined} and the corresponding estimates in Corollary \ref{estsfollow}, we have that for each $z \in L^2(0,T;L^2(\Omega))$, the set $\mathscr F(z) \neq \emptyset$. Moreover, for each element $\zeta \in \mathscr F(z)$, via \ref{F_welldefined}, we have that 
$$\zeta \in L^2(0,T;V), \ \ \text{and} \ \ \zeta_t \in L^2(0,T;V'),$$ with associated estimates in \eqref{preest1}--\eqref{preest3}.
\vskip.2cm
\noindent \underline{Step I.} First, we show that the correspondence is convex- and closed-valued,  i.e., $\mathscr F(z)$ is convex and closed, for each $z \in  L^2(0,T;L^2(\Omega))$. Thus let $z \in  L^2(0,T;L^2(\Omega))$ and let $\zeta_1, \zeta_2 \in \mathscr F(z)$. 
This means that  the pairs $\big(p_i, \zeta_i = [c_0\mathbf I+B]p_i\big)$ satisfy the definition provided in Definition \ref{WeakSol} for $i = 1,2$, namely for every $q \in L^2(0,T; V)$, we have
\begin{equation}\label{varform11}
\int_0^T \big\langle [c_0\mathbf I  + B]p_i'(t), q(t) \big \rangle \ dt + \int_0^T A[p_i(t), q(t);z(t)]\ dt = \int_0^T \langle S(t), q(t) \rangle \ dt
\end{equation}
and $\big[(c_0\mathbf I  + B)p_i\big](0) = d_0$.

Since the problem is {\em linear} in $k(z)$ with $A[p,q;z]=\big(k(z)\nabla p, \nabla q)_{\Omega}$, convexity in the weak form of solutions and initial conditions is immediate by taking the appropriate linear combination of the above equalities. For the associated inequalities in \eqref{preest1}--\eqref{preest3}, the convexity of norms is sufficient. Indeed, we show this for \eqref{preest1}: Suppose for $p_i(z)$, $i=1,2$ we have:
\begin{align}\label{refthis}c_0||p_i(z)||_{L^{\infty}(0,T;L^2(\Omega))}^2+||B^{1/2}p_i(z)||^2_{L^{\infty}(0,T;L^2(\Omega))}&+\|p_i(z)\|_{L^2(0,T;V)} \nonumber  \\ &\le C[||d_0||^2_{L^2(\Omega)}+ ||S||^2_{L^2(0,T;V')}].\end{align}
Then $$||\alpha p_1+(1-\alpha)p_2||^2 \le \alpha||p_1||^2+(1-\alpha)||p_2||^2,$$ and hence by multiplying the inequality \eqref{refthis} by $\alpha$ when $i=1$, and again multiplying by $(1-\alpha)$ when $i=2$, then adding the results, yields that the function $\alpha p_1+(1-\alpha)p_2$ satisfies \eqref{placeholder1} for any $\alpha \in [0,1]$ {\em with the same constant $C$ associated to the RHS}.

 To show that $\mathscr F(z)$ is closed, consider a sequence $\zeta_n \in \mathscr F(z)$ such that $\zeta_n \to \zeta \in L^2(0,T;L^2(\Omega))$. Using the invertibility of the linear operator $[c_0\mathbf I+B]$ on $L^2(0,T;L^2(\Omega))$, we have that $p_n \to (c_0\mathbf I+B)^{-1}\zeta \in L^2(0,T;L^2(\Omega))$, so we let ~$p =  (c_0\mathbf I+B)^{-1}\zeta \in L^2(0,T;L^2(\Omega))$. Since the pair $(p_n,\zeta_n)$ satisfies the estimates \eqref{preest1}--\eqref{preest3}, we know that $p_n$ is uniformly-in-$n$ bounded in $L^2(0,T;V)$. Therefore $p_n$ has a weakly convergent subsequence, whose limit is identified with $p$ by uniqueness of limits. This yields that $p$ lies in $L^2(0,T;V)$.
 
Passing to the limit, then, in the weak form in Definition \ref{WeakSol} is immediate, since the problem is linear in $k(z)$. 
The estimates in \eqref{preest1}--\eqref{preest3} on $p(z)$ and $\zeta(z)=[c_0\mathbf I+B]p(z)$ follow from weak lower semicontinuity of norms, since each $(p_n,\zeta_n)$ satisfies them by hypothesis. Lastly, obtaining the initial condition is immediate, since each $\zeta_n$ corresponds to a solution with the same initial condition $\zeta_n(t=0)=d_0$. The estimates on solutions in \eqref{preest1}--\eqref{preest3} ensure that $\zeta_n, \zeta \in H^1(0,T;V')$, and hence we have that $\zeta(t=0)=d_0$ as the limit point of $\zeta_n(0)$ (in the sense of $V'$).

 \vskip.2cm
 \noindent \underline{Step II.} Next, we show the sequential criterion for UHC of the correspondence $\mathscr F$, as in Theorem \ref{sequentialcrit}.  To that end, let $\{(z_n, \zeta_n)\} \subseteq \mathscr G(\mathscr F)$. Suppose further that $z_n \to z \in L^2(0,T;L^2(\Omega))$.  We want to conclude that $\zeta_n$ has a (strong) limit point $\zeta \in \mathscr F(z)$. 
 
 First, by Assumption \ref{Assumpk}, the function $k(\cdot)$ considered as Nemytskii operator, has the property that $k(z_n) \to k(z) \in L^2(0,T;L^2(\Omega))$. Now, since $\zeta_n \in \mathscr F(z_n)$, for the unique $p_n = [c_0\mathbf I+B]^{-1}\zeta_n$~
 we have by definition of $\mathscr F$ the estimate \eqref{preest1} and that the weak form of the equation is satisfied (as in Definition \ref{WeakSol}). The estimate \eqref{preest1}  (with fixed RHS in terms of data) yields a uniform-in-$n$ bound on 
 $$||p_n||_{L^2(0,T;V)},~~||p_n||_{L^{\infty}(0,T;L^2(\Omega))},~~||B^{1/2}p_n||_{L^{\infty}(0,T;L^2(\Omega))}.$$
 From the bound on $p_n$ in $L^2(0,T;V)$ we extract a weak subsequential limit point, i.e., $p_{n_k} \rightharpoonup p \in L^2(0,T;V).$ From this and the continuity of $[c_0\mathbf I+B] \in \mathscr L(L^2(0,T;L^2(\Omega))),$ we obtain immediately that ~$\zeta_{n_k} =[c_0\mathbf I+B]p_{n_k} \rightharpoonup [c_0\mathbf I+B]p.$ We define this latter quantity as $\zeta \equiv [c_0\mathbf I+B]p,$ and hence $\zeta_{n_k} \rightharpoonup \zeta.$ In addition, the estimate \eqref{preest3} in the definition of $\mathscr F$ and the uniqueness of limits ensure that (perhaps passing to a further subsequence with the same label)~
 $\zeta_{n_k} \rightharpoonup \zeta \in H^1(0,T; V').$
  We want to show that $\zeta \in \mathscr F(z)$, and this is accomplished by passing with the limit on the subsequence $n_k$ in the weak formulation \eqref{varform1}. To that end, let us again consider the weak form evaluated on $n_k$, and restrict our spatial test functions to $q \in L^2(0,T;V) \cap L^{\infty}(0,T;W^{1,\infty}(\Omega))$:
\begin{equation}\label{thisoneagain}
\int_0^T \big\langle \zeta_{n_k}'(t), q(t) \big \rangle \ dt + \int_0^T A[p_{n_k}(t), q(t);z_{n_k}(t)]\ dt = \int_0^T \langle S(t), q(t) \rangle \ dt
\end{equation}
  Limit passage on the first term on the LHS is immediate identifying weak limits in this weak form. For the second term, more care must be taken. Consider:
  \begin{equation}\label{duh}\int_0^T\big(k(z_{n_k})\nabla p_{n_k},\nabla q(t)\big)dt = \int_0^T\big([k(z_{n_k})-k(z)]\nabla p_{n_k},\nabla q(t)\big)dt +\int_0^T(k(z)\nabla p_{n_k},\nabla q(t)) dt.\end{equation}
  The first term on the RHS is handled through the Nemytskii property of $k(\cdot)$:
  \begin{align*} \int_0^T([k(z_{n_k})-k(z)]\nabla p_{n_k},q(t)) dt \le &~ C(||q||_{L^{\infty}(0,T;W^{1,\infty}(\Omega))})||k(z_{n_k})-k(z)||_{L^2(0,T;L^2(\Omega))}|| p_{n_k}||_{L^2(0,T;V)}\\
  \le & ~C(q,||p||_{L^2(0,T;V)})||k(z_{n_k})-k(z)||_{L^2(0,T;L^2(\Omega))} \to 0,
  \end{align*}
  by the uniform bound on $p_{n_k}$ in $L^2(0,T;V)$. 
  Convergence of the second term in \eqref{duh} is immediate, since by the boundedness of $k$ we have $k(z)\nabla q \in L^2(0,T; L^2(\Omega))$; thence, $\nabla p_{n_k} \rightharpoonup \nabla p \in L^2(0,T;L^2(\Omega))$. 
  
  Thus, we have shown that for $q \in L^2(0,T;V) \cap L^{\infty}(0,T;W^{1,\infty}(\Omega))$ 
 $$ \int_0^T(k(z_{n_k})\nabla p_{n_k},\nabla q(t)) dt \to \int_0^T(k(z)\nabla p,\nabla q(t)) dt,$$
 and hence, passing to the limit as $k\to \infty$ in \eqref{thisoneagain} we obtain for $\zeta=\zeta(z)$ the identity
   \begin{equation}\int_0^T\langle\zeta_t,q\rangle dt +\int_0^T(k(z)\nabla p,\nabla q(t))dt = \int_0^T\langle S,q(t)\rangle dt \end{equation}
   for all $q \in L^2(0,T;V) \cap L^{\infty}(0,T;W^{1,\infty}(\Omega))$, the latter being dense in $L^2(0,T;V)$. Thus we have shown that $(\zeta(z),p(z))$ satisfies the weak form of the pressure equation and hence we have constructed a weak solution $(\zeta(z),p(z))$ for $z \in L^2(0,T;L^2(\Omega))$. The requisite estimates in the definition of $\mathscr F$ (\eqref{placeholder1}--\eqref{placeholder3}) hold immediately on the weak subsequential limit points $(\zeta(z),p(z))$ (in the relevant topologies) by weak lower semicontinuity of norms. Obtaining the initial condition is also immediate from the definition of $\mathscr F$. Hence $\zeta_n$ has a {\em weak subsequential limit point $\zeta \in \mathscr F(z)$}. 
   
   To conclude the UHC property of the correspondence $\mathscr F$, we must improve the convergence of $\zeta_{n_k} \to \zeta$ to be strongly in $L^2(\Omega)$. This is done via the Lions-Aubin compactness theorem (see, for instance, \cite{showmono}). In addition to the estimate \eqref{preest1} for the sequence $p_{n_k}$, we obtain two additional uniform-in-$k$ estimates from continuity of $B: V \to H^1(\Omega)$ and from satisfying the weak form of the pressure equation, namely:
   \begin{align}
   ||\zeta_{n_k}||_{L^2(0,T;H^1(\Omega))}^2 =&~ ||c_0p_{n_k}+Bp_{n_k}||_{L^2(0,T;H^1(\Omega))}^2 \lesssim ||p||_{L^2(0,T;V)}^2\\
   \|[\zeta_{n_k}]'\|_{L^2(0,T;V')} =&~  \|[(c_0\mathbf I  + B)p_{n_k}]'\|_{L^2(0,T;V')}^2 \lesssim   \|p\|_{L^2(0,T;V)}^2 + \|S\|_{L^2(0,T;V')}^2
\end{align}
By possibly passing to a further subsequence $n_{k_m}$ (not affecting the previous steps in establishing the weak solution or associated estimates), we improve the convergence of $\zeta_{n_{k_m}} \to \zeta \in L^2(0,T;L^2(\Omega))$. Applying Theorem \ref{sequentialcrit}, we obtain that $\mathscr F: L^2(0,T;L^2(\Omega)) \twoheadrightarrow L^2(0,T;L^2(\Omega))$ is UHC as well as compact-valued. Subsequently, from Theorem \ref{theoneweuse}, we have that $\mathscr F$ is a closed correspondence. 
   
   \vskip.2cm
\noindent \underline{Step III.} Lastly, to invoke the Bohnenblust-Karlin Fixed Point Theorem, we must show that the range of $\mathscr F$ is relatively compact in $L^2(0,T;L^2(\Omega))$. But, as in the previous step, this will follow from the Lions-Aubin compactness criterion. Indeed,  for any $\zeta \in \mathscr R(\mathscr F)$, $\zeta$ corresponds to a weak solution satisfying the estimate \eqref{preest1}, and subsequently \eqref{preest2}--\eqref{preest3}. In particular, we obtain for any such $\zeta$ there is an associated $p=[c_0\mathbf I+B]^{-1}\zeta$ such that:
\begin{align}
||\zeta||^2_{L^2(0,T;H^1(\Omega))} \le C||p||^2_{L^2(0,T;V)} \le C\big[||d_0||^2_{L^2(\Omega)}+||S||_{L^2(0,T;V')}^2] \\
\|\zeta'\|_{L^2(0,T;V')} \le C\big[  \|p\|_{L^2(0,T;V)} + \|S\|_{L^2(0,T;V')}\big] \le C\big[||d_0||^2_{L^2(\Omega)}+||S||_{L^2(0,T;V')}^2\big]
\end{align}
A subset of $L^2(0,T;L^2(\Omega))$ which is bounded as in the previous two estimates is relatively compact by the Lions-Aubin criterion, and hence $\zeta \in \mathscr R(\mathscr F)$ lies in a compact set. This is the final hypothesis to be satisfied for applying the fixed point result, Theorem \ref{BK}. 

Applying the fixed point theorem yields the existence of a function $z \in L^2(0,T;H^1(\Omega)) \cap H^1(0,T;V')$ and an associated weak solution $(\zeta(z),p(z))$ for which $z \in \mathscr F(z)$.
\end{proof}

\begin{remark} We again note that, owing the presence of the nonlinearity, regularity of the solution $\zeta$---in particular of $\nabla \cdot \bu$---needs to be better than $L^2(0,T;L^2(\Omega))$. This is because we must obtain compactness in $\zeta$ to utilize the Nemytskii property of $k(\cdot)$. Moreover, if $d_0 \in V'$ only, this would preclude our ability to obtain such regularity, as this would seem to lower the evolution of $Bp=\nabla \cdot \bu$ to the regularity of $V'$.
\end{remark}

Now we proceed to obtain a weak solution as in Definition \ref{zerosol} to the original problem, which is stated in strong form in Section \ref{translation} as \eqref{abstractform**}. By applying the fixed point result from the previous section, Theorem \ref{th:supporting}, with appropriately scaled/modified data $S, d_0, k(\cdot)$ we will finally prove the main result Theorem \ref{th:main}.

\begin{proof}[Proof of Theorem \ref{th:main}] Considering the strong form of the nonlinear dynamics in \eqref{abstractform**}, we apply the translation as in Section \ref{translation}. For the new variable $\bw=\bu-\bu_{\bF}$, where $\bu_{\bF}$ is defined by \eqref{ellipticsolver}, we obtain the system \eqref{abstractform***}, which we restate here for clarity:
\begin{equation}
\begin{cases}
\cE(\bw)=-\nabla p  &~ \in L^2(0,T; \mathbf V')\\
[c_0 p + \nabla \cdot \bw]_t-\nabla \cdot [k_{\bF}(c_0p+\nabla \cdot \bw )\nabla p]=S+\nabla \cdot \bu_{\bF,t}&~ \in L^2(0,T;V')\\
c_0 p(0) + \nabla\cdot\bw(0)=d_0-\nabla \cdot \bu_{\bF}(0)&~ \in L^2(\Omega).
\end{cases}
\end{equation}
Under the regularity assumptions on $\bF$, we have $\bu_{\bF} \in L^2(0,T; \mathbf H^2(\Omega) \cap \bV)\cap H^1(0,T; \bV)$. Hence, in the above strong form, we have $\nabla \cdot \bu_{\bF,t} \in L^2(0,T;L^2(\Omega)) \subset L^2(0,T;V')$. Additionally, $\nabla \cdot \bu_{\bF} \in C([0,T]; L^2(\Omega))$ \cite{evans,showmono}, so we can extract its time trace at $t=0$. Finally, since $\nabla\cdot \bu_{\bF} \in C([0,T];L^2(\Omega))$, we observe $k_{\bF}(\cdot) = k(\cdot + \nabla \cdot \bu_{\bF})$ satisfies the hypotheses of Assumption \ref{Assumpk} since $k(\cdot)$ is assumed to satisfy them. We can then reduce the above system to a version of \eqref{refthisone}, where the corresponding ``data" satisfies the hypotheses of Theorem \ref{th:supporting}, and we can apply the result of the fixed point theorem. 

Doing so, we obtain a function $p \in L^2(0,T;V)$ that satisfies: 
\begin{itemize}
\item For every $q \in L^2(0,T; V)$, 
\begin{equation*}
\int_0^T \big\langle[(c_0\mathbf I  + B)p]', q\big\rangle \ dt + \int_0^T \big(k(c_0p+Bp+\nabla \cdot \bu_{\bF})\nabla p, \nabla q)dt = \int_0^T \langle S, q\rangle dt+\int_0^T (\nabla \cdot \bu_{\bF,t}, q) dt.
\end{equation*}
\item $\zeta=c_0p+Bp$ a.e. $t$ and a.e. $\mathbf x$.
\item $\big[(c_0\mathbf I  + B)p\big](0) = d_0$ in the sense of $V'$. 
\item The following estimates hold:
\begin{align*}
c_0||p(t)||_{L^2(\Omega)}^2+||B^{1/2}p(t)||^2_{L^2(\Omega)}+\|p\|_{L^2(0,T;V)} \lesssim &~ ||d_0||^2_{L^2(\Omega)}+||\nabla \cdot \bu_{\bF}(0)||^2_{L^2(\Omega)}+ \|S\|^2_{L^2(0,T;V')}  \\[.2cm]
\|[(c_0 \mathbf I + B)p]'\|_{L^2(0,T;V')} \lesssim &~  \|p\|_{L^2(0,T;V)} + \|S\|_{L^2(0,T;V')}+||\nabla \cdot \bu_{\bF,t}||_{L^2(0,T;V')} \\
\|(c_0\mathbf I  + B)p\|_{L^2(0,T;H^1(\Om))} \lesssim &~ \|p\|_{L^2(0,T;V)}. 
\end{align*}
\end{itemize}

From this we can obtain a $\bw \in L^2(0,T; \mathbf  H^2(\Omega) \cap \bV)$ and then a $\bu=\bw+\bu_{\bF}$, resulting in  
$$\bu = \cE^{-1}(-\nabla p + \bF).$$ This $\bu$ has the necessary property that it can be identified via the relation 
$$\nabla \cdot \bu = Bp+\nabla \cdot \bu_{\bF}$$
in a point-wise sense. The above is sufficient to conclude the weak form of the elasticity equation, \eqref{weakform}. Using the fact that 
$||\bu_{\bF}||_{\bV} \lesssim ||\bF||_{\bV'}$ and the embedding $H^1(0,T;L^2(\Omega)) \hookrightarrow C([0,T];L^2(\Omega))$, we re-interpret the estimates above as
\begin{align}\label{final1}
c_0||p||_{L^{\infty}(0,T;L^2(\Omega))}^2+\|p\|_{L^2(0,T;V)}^2 \lesssim &~ ||d_0||^2_{L^2(\Omega)}+DATA\big|_0^T  \\[.2cm]
||\bu||_{L^2(0,T;\mathbf H^2\cap \bV)}^2 \lesssim&~||p||_{L^2(0,T;V)}^2+||\bF||^2_{L^2(0,T;\mathbf L^2(\Omega))}\\
\|[c_0p+\nabla \cdot \bu]'\|_{L^2(0,T;V')} \lesssim &~  \|p\|_{L^2(0,T;V)} + DATA\big|_0^T \label{final3}\\
\|c_0p+\nabla \cdot \bu\|_{L^2(0,T;H^1(\Om))} \lesssim &~ \|p\|_{L^2(0,T;V)}. \label{final4}
\end{align}
(Note that in these final estimates for the original problem we have omitted references to the $B$ operator, in doing so, discarding the information on $B^{1/2}p \in L^{\infty}(0,T;L^2(\Omega)).$)

This finally concludes the proof of Theorem \ref{th:main}.
\end{proof}

\subsection{Proof of Theorem \ref{th:main} - Incompressible Constituents, $c_0=0$}
Consider a sequence or real numbers $\{c_0^m\}_1^{\infty}$ such that $c_0^m <1$, for all $m \geq 1$, and  $c_0^m \searrow 0$ as $m \to \infty$. To each $c_0^m$ we attribute a particular weak solution $(p^m,\bu^m) \in L^2(0,T;V) \times L^2(0,T;\bV)$ to \eqref{abstractform**} (i.e., a solution in the sense of Definition \ref{zerosol}) with a fixed initial condition $d_0 \in L^2(\Omega)$ for all $m$ (the intended initial condition when $c_0=0$). Such a solution has, by construction in the proof of Theorem \ref{th:main},  the energy estimate:
$${c_0^m}||p^m||_{L^{\infty}(0,T;L^2(\Omega))}^2 +k_{1}\int_0^T ||\nabla p^m||^2d\tau \lesssim ||d_0||^2+ DATA\big|_0^T.$$
Note, the bound on the RHS is uniform in $m$. This provides a weak-* subsequential limit point for $\sqrt{c_0^m}p^m \in L^{\infty}(0,T;L^2(\Omega))$), and a weak subsequential limit point labeled $p$ for the sequence $p^m \in L^2(0,T;V)$. From the elasticity equation, we infer that the associated $\bu^m \in L^2(0,T;\mathbf H^2(\Omega)\cap \bV)$ has the bound
$$\int_0^T||\bu^m||^2_{\mathbf H^2(\Omega)\cap \bV}dt \lesssim \int_0^T||\cE(\bu^m)||^2dt \lesssim \int_0^T||\nabla p^m||^2dt+\int_0^T||\bF||_{\mathbf L^2(\Omega)}^2dt,$$ and taken in conjunction with the estimate above, leads to a uniform-in-$m$ estimate. To the sequence $\bu^m$ we also associate a weak subsequential limit point (with the same subsequence associated to $p$, perhaps upon re-indexing):
$$\bu^m \rightharpoonup \bu \in L^2(0,T;\bV).$$

From the construction of the weak solution, we have that $p^m$ satisfies for all $q \in L^2(0,T;V)$:
\begin{equation}\label{thisweak} \int_0^T\langle \zeta_t^m(t), q(t) \rangle dt + \int_0^T(k(\zeta^m(t))\nabla p^m(t),\nabla q(t))dt = \int_0^T\langle S(t),  q(t)\rangle dt,\end{equation}
with $\zeta^m = c_0^mp^m+\nabla \cdot \bu^m=[c_0^m\mathbf I + B]p^m+\nabla \cdot \bu_{\bF} \in L^2(0,T;L^2(\Omega))$ and the associated bound
\begin{align} ||\zeta^m||_{L^2(0,T;L^2(\Omega))} =&~ ||c_0^mp^m+\nabla \cdot \bu^m||_{L^2(0,T;L^2(\Omega))} \\ \nonumber
\le &~ c_0^m||p^m||_{L^2(0,T;L^2(\Omega))}+||\nabla \cdot \bu^m||_{L^2(0,T;L^2(\Omega))} \\\nonumber
\le & ~C\big[||p||_{L^2(0,T;V)} + ||\bu||_{L^2(0,T,\bV)}\big]\\
\lesssim & ~||d_0||^2+DATA\big|_{0}^T+||\bF||^2_{L^2(0,T;\mathbf L^2(\Omega))}
\end{align}
by lower weak semicontinuity of the norm, Poincar\'e, and the boundedness of the sequence $\{c_0^m\}$. Hence there is $\zeta \in L^2(0,T; L^2(\Omega))$ so that 
$$\zeta^m \rightharpoonup \zeta \in L^2(0,T;L^2(\Omega)).$$

As  in the construction of the solution, satisfying the equation \eqref{thisweak} is also sufficient to deduce that $\zeta^m_t \in L^2(0,T;V')$ (with associated uniform-in-$m$ bound as above in \eqref{final3}).  Again, by such a uniform-in-$m$ bound on $\zeta_t^m \in L^2(0,T;V')$ we can extract a weak subsequential limit in that space, and identify in the standard way for the distributional derivative \cite{kesavan} $$\zeta_t^m \rightharpoonup \zeta_t \in L^2(0,T;V').$$
Additionally, the sequence $p^m$ itself is bounded as
$$||p^m||_{L^2(0,T;L^2(\Omega))} \le C||p^m||_{L^2(0,T;V)} \lesssim ||p||_{L^2(0,T;V)}.$$
This implies that the sequence $[c_0^mp^m] \to 0 \in L^2(0,T;L^2(\Omega))$, as well as, for any $\phi \in V$ and $f \in \mathscr D(0,T)$,
$$c_0^m\int_0^T(p^m(t),\phi)_{\Omega}f'(t) dt \to 0,$$ as $c_0^m \to 0$.
From which we deduce by the definition of the distributional derivative in $t$ (see Remark \ref{equivvarforms}) that 
\begin{equation}\label{thiszero}c_0^m\int_0^T\langle p_t^m(t), q(t)\rangle_{V'\times V} dt \to 0,~~\forall q \in L^2(0,T;V).\end{equation}

Now, it is immediate from the previously established convergences that $\zeta^m = c_0^mp^m+\nabla \cdot \bu^m \rightharpoonup \nabla \cdot \bu \in L^2(0,T;L^2(\Omega))$, and by uniqueness of limits, we must have that $\zeta=\nabla \cdot \bu \in L^2(0,T;L^2(\Omega))$, possibly passing to a subsequence. Moreover,  since $\zeta_t^m \in L^2(0,T;V')$ with $c_0^mp^m_t \rightharpoonup 0$, we obtain that 
$$\int_0^T\langle \zeta_t^m,q \rangle dt \to \int_0^T \langle \zeta_t,q\rangle,~~~~\int_0^T\langle \zeta_t^m,q \rangle dt  = \int_0^T \big[\langle c_0^mp^m_t,q\rangle + \langle \nabla \cdot \bu_t^m, q\rangle\big]dt.$$
But we know, again by uniqueness of limits, since we have \eqref{thiszero}, that
$$\int_0^T\langle \zeta^m_t,q\rangle \to  \int_0^T\langle  \zeta_t,q\rangle dt = \int_0^T\langle \nabla \cdot \bu_t,q\rangle dt.$$

Lastly, we need to show the nonlinear term exhibits convergence; namely, we want to show
\begin{equation}\label{finalconv} \int_0^T(k(\zeta^m)\nabla p^m , \nabla q)dt \to \int_0^T(k(\nabla \cdot \bu)\nabla p , \nabla q)dt.\end{equation}
This will follow immediately as in Step II in the proof of Theorem \ref{th:supporting}, using the Nemytskii property of $k(\cdot)$. Indeed, we first improve the bounds on $\zeta^m$:
\begin{align} \nonumber ||\zeta^m||_{L^2(0,T;H^1(\Omega))} =&~ ||c_0^mp^m+\nabla \cdot \bu^m||_{L^2(0,T;H^1(\Omega))} \\ \nonumber
\lesssim&~c_0^m||p^m||_{L^2(0,T;V)}+||\bu^m||^2_{\mathbf H^2(\Omega)}\nonumber \\
\lesssim&~||p||_{L^2(0,T;V)}+||\bu||_{L^2(0,T;\bV\cap \mathbf H^2(\Omega))}
\end{align}
This again results in uniform-in-$m$ boundedness, and we note that (perhaps on a subsequence), again by The Lions-Aubin criterion and the uniform boundedness of $\{\zeta^m\} \subset L^2(0,T;\mathbf H^1(\Omega))$ and $\{\zeta^m_t\} \subset L^2(0,T;V')$, we can improve the convergence of $\zeta^m \to \zeta \in L^2(0,T; L^2(\Omega))$ to {\em strong} convergence. At this point, $k(\zeta^m) \to k(\zeta) \in L^2(0,T;L^2(\Omega))$, and we invoke the identification $\zeta=\nabla \cdot \bu$. From this, \eqref{finalconv} follows.

Hence, upon taking  the limit as $c_0^m \to 0$ in \eqref{thisweak}, we obtain that the subsequential limit point $(p,\bu) \in L^2(0,T; V \times \bV)$ satisfies:
\begin{equation}\label{thisweak0} \int_0^T\langle \nabla\cdot \bu_t(t),q \rangle dt + \int_0^T\big(k(\zeta(t))\nabla p(t),\nabla q\big)f(t)dt = \int_0^T\langle S(t),  q\rangle  f(t) dt,\end{equation}
for any $q \in L^2(0,T;V)$. 

As in the $c_0>0$ case, we recover the elasticity equation faithfully, so
$$\cE(\bu)=-\nabla p+\mathbf F,~~\text{a.e.}~\mathbf x, t,$$
which of course yields the weak form in Definition \eqref{WeakSol}.

Finally, by the equation, we again have that $\nabla \cdot \bu \in H^1(0,T;V') \cap L^2(0,T;H^1(\Omega))$, so $\nabla \cdot \bu \in C([0,T]; V')$, which permits the initial condition for the quantity $\nabla \cdot \bu (0)=d_0$ in the $V'$ sense, though $d_0 \in L^2(\Omega)$. This ensures that $(p,\bu) \in L^2(0,T; V \times \bV)$ is in fact a weak solution with $c_0=0$, in the sense of Definition \ref{zerosol} (see Remark \ref{dataremark2}).

\section{Uniqueness -  Proof of Theorem \ref{th:main2}}\label{Uniqueness}
In this section we divide our considerations into two approaches, depending on which terms have point-wise-in-$t$ control: (i) the first one considers a weak solution to the full dynamical system in both dependent variables $(\bu,p)$; (ii) the second approach considers the reduced system phrased in terms of $p$ and $Bp$. There are subtle differences in the approaches and in the requisite hypotheses to obtain unique solutions. We point out that---since we do not construct strong solutions in this paper---we phrase our results below as: {\em If weak solutions exhibit  additional regularity properties, then among such weak solutions there is uniqueness}. We recall the space of weak solutions (for given data $d_0,S,\bF,k(\cdot)$) with the additional regularity $p_t \in L^2(0,T;L^2(\Omega))$ denoted by $\mathcal W_T$---see \eqref{W_T}. In particular, our result says: {\em If one weak solution in $\mathcal W_T$ has additional spatial regularity, then all weak solutions in $\mathcal W_T$ are equal to it.}

Let us introduce both types of formal energy identities utilized later in the proof of Theorem \ref{th:main2} for  $c_0\ge 0$. 
Consider $(\bu^1,p^1)$ and $(\bu^2,p^2)$ two weak solutions coming from $\mathcal W_T$. Then, we subtract the weak forms of the equations as in Definition \ref{WeakSol} and test the pressure equation with $\overline p = p^1-p^2$ and the elasticity equation with $\overline{\bu}_t=\bu^1_t-\bu^2_t$. The latter is justified, as $\bu^i_t = \mathcal E^{-1}\big(-\nabla p^i_t+\bF_t\big) \in L^2(0,T;\bV)$ since $\bF_t \in L^2(0,T;\bV')$ by our regularity hypotheses on $\bF$ and $\nabla p^i_t \in L^2(0,T;\bV')$ since $p^i_t \in L^2(0,T;L^2(\Omega))$.
 
 This yields the unsimplified identities:
\begin{align}\label{formalfull1}
\int_0^T\langle \cE(\overline\bu),\overline \bu_t\rangle_{\bV'\times \bV} dt-\int_0^T( \overline p, \nabla \cdot \overline \bu_t)_{\Omega} dt = 0 \\ \label{formalfull2}
\int_0^T \big[c_0( \overline p_t, \overline p )_{\Omega} +( \nabla \cdot \overline\bu_t, \overline p)_{\Omega} \big]dt +\int_0^T\big(k(\zeta^1)\nabla p^1-k(\zeta^2)\nabla p^2, \nabla \overline p\big)_{\Omega}dt = 0,
\end{align}
where we denote $\zeta^i=c_0p^i+\nabla \cdot \bu^i$. 

Now, let us consider the formal energy relation for the (partial, omitting the equation for $\bu$) reduced formulation making use of the $B$ operator, again, with no simplifications:
\begin{align}\label{reduced}
\big( [c_0\mathbf I+B]\overline p_t,\overline p \big )_{\Omega}+\big(k(\zeta^1)\nabla p^1-k(\zeta^2)\nabla  p^2, \nabla \overline p)_{\Omega} = 0.\end{align}

We have replaced the $V'\times V$ duality pairings in the pressure equations above through the assumption that $p^i_t \in L^2(0,T;L^2(\Omega))$. Both approaches to uniqueness hinge on the analysis of the nonlinear term. The goal is to apply  a version of Gr\"onwall to the formal estimates. 

Let us begin by estimating directly this nonlinear term above.
\begin{align}
\int_0^T \left([k ^1(t)\nabla p ^1(t)-k ^2(t)\nabla p ^2(t)],\nabla \overline p\right)dt = &~\int_0^T\big([k ^1(t)-k ^2(t)]\nabla p^1 (t),\nabla \overline p(t)\big)dt \\
&+\int_0^T\big(k^2 (t)\nabla \overline p (t),\nabla \overline p (t)\big)dt, \nonumber
\end{align}
where we have used the shorthand $k^i(t) \equiv k(c_0p^i+\nabla \cdot \bu^i)$.
Using the lower bound on the permeability function $0<k_1\le k(\cdot)$ from Assumption \ref{Assumpk} the second term above will serve as dissipation to help with further estimation
$$k_1|| \overline p||^2_{L^2(0,T; V)} \le \int_0^T\big(k^2 (t)\nabla \overline p (t),\nabla \overline p (t)\big)dt.$$

The remaining nonlinear term on the RHS can be estimated in two ways, yielding the two distinct hypotheses. With the supplemental hypothesis that $k \in Lip(\mathbb R)$ with $L_k>0$ the global Lipschitz constant, we have
$$|k^1(t)-k^2(t)| = \big|k(\zeta^1)-k(\zeta^2)| \le L_k|\overline \zeta|,$$
again where $\overline \zeta = \zeta^1-\zeta^2$, and $\zeta^i = c_0p^i+\nabla \cdot \bu^i.$

Using Cauchy-Schwartz, then, we have
$$\int_0^T\big([k^1(t)-k^2(t)]\nabla p^1(t),\nabla \overline p(t)\big)dt  \le L_k\int_0^T||\nabla p^1||_{L^{\infty}(\Omega)} ||\overline \zeta ||_{L^2(\Omega)}||\nabla \overline p||_{L^2(\Omega)}dt.$$

We proceed straightforwardly, retaining the supremum term under the integration:
\begin{align}\nonumber \int_0^T\big([k ^1(t)-k ^2(t)]\nabla p^1 (t),\nabla \overline p (t)\big)dt  \le&~ \int_0^T\big[L_k||\nabla p ^1||_{L^{\infty}(\Omega )}\big] ||\overline\zeta(t)||_{L^2(\Omega )}||\nabla \overline p ||_{L^2(\Omega )}dt  \\\label{way2}
\le &\int_0^T\dfrac{L_k^2}{4\epsilon}||\nabla \overline p (t)||^2_{L^{\infty}(\Omega )}\big[||\overline \zeta(t)||^2_{L^2(\Omega)}\big]dt \\ \nonumber
&+ \epsilon \int_0^T||\nabla \overline p ||^2_{L^2(\Omega )}dt,~~\forall~\epsilon>0.
\end{align}
\begin{remark}
One can also pull the supremum term outside the integral; this is akin to the approach taken in \cite{cao}.
\begin{align}\nonumber \int_0^T\big([k^1(t)-k^2(t)]\nabla p^1(t),\nabla \overline p(t)\big)dt  
\label{way1}
\le &~ \dfrac{L_k^2}{4\epsilon}||\nabla p^1 ||^2_{L^{\infty}(0,T;L^{\infty}(\Omega ))}\int_0^T||\overline \zeta(t)||^2_{L^2(\Omega )}dt \\\nonumber 
&+ \epsilon \int_0^T||\nabla \overline p ||^2_{L^2(\Omega )}dt,~~\forall~\epsilon>0.
\end{align}
\end{remark}

Anticipating the use of Gr\"onwall below, we note that $\nabla p^1 \in L^{\infty}(\Omega)$ is obtained through the Sobolev embeddings in 3-D if for instance, if $p^1 \in H^3(\Omega)$ (or any Sobolev index above $2.5$).

\subsection{Uniqueness for the Full System; $c_0 \ge 0$}
In working with the full system, we can exploit cancellation in the structure of the Biot system to obtain more explicit energy estimates to be used for uniqueness. In this framework, as we shall see, we need to specify the initial displacement $\bu_0 \in \bV$ independently of $c_0p_0$, recalling that $\zeta(0) = d_0 = [c_0p+\nabla \cdot \bu](0).$

Then, we consider \eqref{formalfull1}--\eqref{formalfull2} and add the two equations, cancelling cross-terms on the RHS and simplifying by integration by parts. This yields:
\begin{equation}
a(\overline \bu(t),\overline \bu(t)) +c_0||\overline p(t)||^2+ 2\int_0^t(k(\zeta^1)\nabla p^1-k(\zeta^2)\nabla p^2, \nabla \overline p) dt= e(\overline \bu_0,\overline \bu_0)+c_0||\overline p_0||^2.
\end{equation}
We recall the estimate on a single trajectory:
\begin{equation}
||\bu||^2_{L^{\infty}(0,T;\bV)}+c_0||p||^2_{L^{\infty}(0,T;L^2(\Omega))}+k_1||p||^2_{L^2(0,T;V)} \le C(\bu_0,p_0)+DATA\big|_0^T.
\end{equation}
 The resulting estimate on $(\overline \bu, \overline p)$ as above in \eqref{way2} is
\begin{align}
||\overline{\bu}||^2_{L^{\infty}(0,T;\bV)}+c_0||\overline p||^2_{L^{\infty}(0,T;L^2(\Omega))}+k_1||\overline p||^2_{L^2(0,T;V)}  
\lesssim&~C(\overline{\bu_0},\overline p_0)+\int_0^T\big( [k^1(t)-k^2(t)]\nabla p^1 ,\nabla \overline p\big)dt, \nonumber \\\nonumber
\lesssim&~\dfrac{L^2_k}{4\epsilon}\int_0^T||\nabla \overline p (t)||^2_{L^{\infty}(\Omega )}\big[||\overline \zeta(t)||^2_{L^2(\Omega)}\big]dt \\
&+ \epsilon \int_0^T||\nabla \overline p ||^2_{L^2(\Omega )}dt
\end{align}
We then note that $\overline \zeta = c_0\overline p +\nabla \cdot \overline \bu$, and hence
$$||\overline \zeta||^2 \lesssim c_0^2||\overline p||^2+||\overline \bu||_{\bV}^2.$$
Absorbing on the RHS by choosing, e.g., $\epsilon = \frac{k_1}{2}$, we obtain:
\begin{equation}
||\overline{\bu}||^2_{L^{\infty}(0,T;\bV)}+c_0||\overline p||^2_{L^{\infty}(0,T;L^2(\Omega))} \lesssim \dfrac{L^2_k}{k_1}\int_0^T||\nabla \overline p (t)||^2_{L^{\infty}(\Omega )}\big[c_0^2||\overline p||^2+||\overline \bu||_{\bV}^2\big]dt.\end{equation}
If $c_0<1$, then we have
\begin{equation}
||\overline{\bu}||^2_{L^{\infty}(0,T;\bV)}+c_0||\overline p||^2_{L^{\infty}(0,T;L^2(\Omega))} \lesssim \dfrac{L^2_k}{k_1}\int_0^T||\nabla \overline p (t)||^2_{L^{\infty}(\Omega )}\big[c_0||\overline p||^2+||\overline \bu||_{\bV}^2\big]dt\end{equation}
If $c_0>1$, then we have
\begin{equation}
||\overline{\bu}||^2_{L^{\infty}(0,T;\bV)}+c_0||\overline p||^2_{L^{\infty}(0,T;L^2(\Omega))} \lesssim \dfrac{L^2_kc_0}{2k_1}\int_0^T||\nabla \overline p (t)||^2_{L^{\infty}(\Omega )}\big[c_0||\overline p||^2+||\overline \bu||_{\bV}^2\big]dt\end{equation}
From here, we may invoke $L^2$-kernel version of Gr\"onwall as in \cite[Theorem 9]{dragonmir}, and uniqueness of solutions is deduced in the standard way.

\subsection{Reduced Equation Uniqueness; $c_0>0$}

In this section we consider working with the reduced equation directly. We assume only that $d_0 \in L^2(\Omega)$, forgoing any assumptions on $\bu(t=0)$. As we will see, we need to assume $c_0>0$ as well.

So, given $\bF$ and $S$ as above, let us consider two weak solutions $p^i(t) \in L^2(0,T;V)\cap H^1(0,T;L^2(\Omega))$ (this follows, for instance, if $(\bu,p) \in \mathcal W_T$ and the problem is reduced through the $B$ operator) to
$$[c_0I+B]p_t-\nabla \cdot k(\zeta)\nabla p = S+\nabla \cdot \bu_{\bF,t} \in L^2(0,T;V'),$$
using the notation from Section \ref{translation}.
We will denote $\zeta=c_0p+Bp+\nabla \cdot \bu_{\bF}$ here for the fluid content.

\begin{remark}[Weakening hypotheses] Here, the main regularity we need is to be able to interpret the pairing 
$\langle [c_0\mathbf I+B]p_t,p\rangle$ in some sense. The challenge is that the properties of $B$ in both $V$ and $V'$ are not clear (e.g., self-adjointness), and for $p \in L^2(0,T;V)$, it is not clear that $Bp\in L^2(0,T;V)$.
\end{remark}

Let $\overline p=p^1-p^2$ as before, and hence $\overline \zeta=c_0\overline p+B\overline p$. Then the straightforward energy relation in \eqref{reduced} simplifies to
\begin{align*}
\frac{1}{2}\dfrac{d}{dt}\big[c_0||\overline p||^2+(B\overline p,\overline p)\big]+(k(\zeta^1)\nabla p^1-k(\zeta^2)\nabla p^2, \nabla \overline p) =0.
\end{align*}
Add and subtract, anticipating using the Lipschitz property of $k$:
\begin{align*}
c_0||\overline p(t)||^2+||B^{1/2}\overline p(t)||^2+2\int_0^t(\nabla p_1[k(\zeta^1)-k(\zeta^2)],\nabla \overline p)+(k(\zeta^2)\nabla \overline p, \nabla \overline p)dt =(\overline{d_0},\overline p(0))\end{align*}
Since $c_0>0$, we can recover $p(0)=p_0=[c_0\mathbf I+B]^{-1}d_0$. 
Estimating as in the previous section and invoking the assumptions on $k(\cdot)$, we obtain
\begin{align*}
c_0||\overline p||_{L^{\infty}(0,T;L^2(\Omega))}^2+||B^{1/2}\overline p||_{L^{\infty}(0,T;L^2(\Omega))}^2&+k_1||\overline p||_{L^2(0,T;V)}^2 \\
\lesssim&~ ||\overline{d_0}||^2+\dfrac{L^2_k}{k_1}\int_0^T|| \nabla p^1 ||_{L^{\infty}(\Omega)}^2~ ||\overline \zeta ||_{L^2(\Omega)}^2 d\tau.
\end{align*}
To proceed as before with Gr\"onwall, it is imperative here that $c_0>0$ since we do not know that $B$ or $B^{1/2}$ is coercive. We then estimate $\overline \zeta$ carefully:
$$||\overline \zeta|| = ||c_0\overline p+B\overline p|| \le C||p||,$$ where all norms are taken in the $L^2(\Omega)$ sense. Simplifying the above inequality, and invoking this estimate, we obtain
\begin{equation}
c_0||\overline p||^2_{L^{\infty}(0,T;L^2(\Omega)} \lesssim ||\overline{d_0}||_{L^2(\Omega)}^2+\dfrac{L_k^2}{k_1}\int_0^T|| \nabla p^1 ||_{L^{\infty}(\Omega)}^2~ ||\overline p ||_{L^2(\Omega)}^2 dt.
\end{equation}
Since $p \in H^1(0,T;L^2(\Omega))$ in this case, $p \in C([0,T];L^2(\Omega))$ and 
at this point, Gr\"onwall can be applied as before to obtain uniqueness in $p$ a.e. $t$ and $\mathbf x$, which can then be transferred through the elasticity isomorphism $\cE$ to $p$. This results in uniqueness of the weak solution $(\bu,p) \in \mathcal W_T$.

\section{Appendix A: Multivalued Fixed Point}\label{app1}
We begin with a handful of definitions and straightforward theorems that will be relevant to the fixed point we are using in the construction of weak solutions. All of these considerations are taken from \cite{guide}.

The basic setting considers $\phi: X \twoheadrightarrow Y$ as a correspondence, where, for each $x \in X$, $\phi(x)$ represents a subset of $Y$. (We do not use the equivalent point of view that $\phi: X \to 2^Y$.) The $\twoheadrightarrow$ notation indicates that $\phi$ need not be a function, but is thought of as a ``multi-valued function."

\begin{definition}[Notions of Closedness and Compactness]
A correspondence $\phi: X \twoheadrightarrow Y$ between topological spaces  is {\em closed-valued} if $\phi(x)$ is a closed set for each $x \in X$. The analogous definition is used for a {\em compact-valued} correspondence. 

A correspondence $\phi: X \twoheadrightarrow Y$ between topological spaces is {\em closed} (or {\em has a closed graph}) if 
$$\mathscr G(\phi) \equiv \{ (x,y) \in X \times Y ~:~ y \in \phi(x)\}$$ 
is closed as a subset of $X \times Y$. 
\end{definition}

\begin{definition}
A correspondence $\phi: X \twoheadrightarrow Y$ between topological spaces is called {\em upper hemicontinous} (or UHC) at the point $x \in X$ if for every neighborhood $U \ni x$ there is a neighborhood $V \ni x$ such that
$$z \in V \implies \phi(z) \subseteq U.$$ We say that $\phi$ is UHC on $X$ if it is UHC at each $x \in X$. 
\end{definition}

The next theorem provides the relationship between graph closedness and UHC. (We do not explicitly use this version in the body of the paper.)
\begin{theorem}
Suppose $\phi: X \twoheadrightarrow Y$ is closed-valued. If $\phi$ is UHC at $x$, then for all $x_n \in X$, $y \in Y$, and $y_n \in \phi(x_n)$
$$x_n \to x~\text{ and }~y_n \to y~~\implies ~~y \in \phi(x).$$
If $\phi$ is closed-valued and the range of $\phi$ is compact, then the converse holds. 
\end{theorem}

Alternatively, the following is the criteria we invoke in the proof of our main result:
\begin{theorem}\label{theoneweuse}
Suppose $\phi: X \twoheadrightarrow Y$ is an UHC correspondence. If $\phi$ is closed-valued (and $Y$ is regular) OR $\phi$ is compact-valued (and $Y$ is Hausdorff), then $\phi$ is closed. 
\end{theorem}

The next theorem is a subtle variation on the previous sequential criteria for upper-hemicontinuity. 
\begin{theorem}\label{sequentialcrit}
Assume that a topological space $X$ is first countable and  $Y$ is metrizable. Then for a correspondence $\phi: X \twoheadrightarrow Y$  and a point $x \in X$ TFAE:
\begin{itemize}
\item $\phi$ is UHC at $x$ and $\phi(x) \subset \subset Y$.
\item If a sequence $\{(x_n,y_n)\}$ in $\mathscr G(\phi)$ satisfies $x_n \to x$ then $\{y_n\}$ has a limit point in $\phi(x)$. 
\end{itemize}
\end{theorem}

Finally, we are in a position to state the multi-valued fixed point theorem employed in our constructions above, the {\em Bohnenblust-Karlin} theorem. Historically, this theorem has been considered as the multi-valued version of the Schauder fixed point theorem. Let us note that fixed point for a correspondence $\phi: X \twoheadrightarrow X$ is simply a point $x \in X$ so that $x \in \phi(x)$.

\begin{theorem}[Bohnenblust-Karlin]\label{BK}
Let $X$ be a nonempty closed convex subset of a locally Hausdorff space, and let $\varphi: X \twoheadrightarrow X$ be a correspondence with closed graph and nonempty convex values. If the range of $\varphi$ is relatively compact (or equivalently, if it is included in a compact set), then the set of fixed points of $\varphi$ is nonempty and compact. 
\end{theorem}

\section{Appendix B: Galerkin Construction for Linear Problem}\label{app2}
Proof of Lemma \ref{F_welldefined}:
\begin{proof} 
Due to Assumption 1.1 on the permeability operator $k$, the following Proposition is immediate.
\begin{proposition}\label{AssumpA} The bilinear form $A[\cdot, \cdot;z(t)]$ satisfies the following properties:
\begin{enumerate}
\item Continuity: $\exists M > 0$ s.t.
$|A(w_1, w_2;z(t))| \leq M \|w_1\|_{V} \|w_2\|_{V}, \ \ \forall w_1, w_2 \in V$, \ a.e. in\  \ $[0,T]$.
\item Coercivity: $A(w,w; z(t))  \geq k_1 \|w\|^2_{V}$, for all $w \in V$. 
\end{enumerate}
\end{proposition}

\noindent \underline{Construction of Approximate Solution:}
We use Galerkin approximations. 
Let $\{w_k(x)\}_{k=1}^{\infty}$ be an orthogonal basis of $V$, and an orthonormal basis in $L^2(\Omega)$. (For example, we can take $\{w_k(x)\}_{k=1}^{\infty}$ to be the complete set of appropriately normalized eigenfunctions for $-\Delta$ in $V$.) 
Let $V_n = \text{span}\{w_1, ... w_n\}$. Note that $V_n$ satisfies the conditions $V_n \subset V_{n+1}$ and $\overline{\cup V_n} = V$. 
We look for solutions of the form:
\begin{equation}\label{pninVn}
p_n(t) = \sum_{k=1}^n d^k_n(t) w_k,
\end{equation}
 where the coefficients $d^k_n(t) \in H^1(0, T)$ for $k = 1, ..., n$. Thus we consider the following finite dimensional problem on $V_n$: 
 
 Determine $p_n \in H^1(0, T; V)$ such that for every $k = {1,2,..,n}$, 
\begin{equation}
\begin{cases}
([c_0\mathbf I  + B]p'_n, w_k)_{L^2(\Omega)} + A[p_n, w_k; z(\cdot)] = \langle S, w_k\rangle, \ \ \text{a.e. in} \ (0,T), \label{Galerkin1}\\
d^k_n(0) = ([c_0\mathbf I  + B]^{-1}d_0, w_k)_{\Omega}, k = 1,2,...,n.
\end{cases}
\end{equation}
If the differential equation in \eqref{Galerkin1} holds for each element of the basis $w_k$, with $ k = 1,2,...,n$, then it also holds for every $w \in V_n$. Moreover, since $(c_0 \mathbf I  + B)p'_n \in L^2(0,T; L^2(\Omega))$, we have from Remark \ref{equivvarforms} 
$$((c_0 \mathbf I + B)p'_n(t), w)_{\Omega} = \langle (c_0\mathbf I  + B)p'_n(t), w\rangle$$
Upon expanding $p_n$,  (\ref{Galerkin1}) becomes 
\begin{equation}
\begin{cases}\ds
M (d^k_n(t))' + \sum_{k=1}^n A[w_l, w_k; z(t)] d^l_n (t) = S^k(t), \    \label{Galerkin2}\\
d^k_n(0) = ([c_0\mathbf I+B]^{-1}d_0, w_k), ~~k = 1,2,...,n,
\end{cases}
\end{equation}
where $$M = ([c_0\mathbf I  + B]w_k, w_k)_{\Omega}, ~\text{ and }~ S^k(t) =  \langle S(t), w_k\rangle,~~  k = 1,2,...,n.$$

Since $(c_0 \mathbf I + B)$ is invertible on $L^2(\Omega)$, we have that $\{[c_0\mathbf I+B]w_k\}_{k=1}^{\infty}$ is linearly independent in $L^2(\Omega)$. Therefore we can find a permutation $\alpha(i)$ of the basis $\{w_k\}$ such that for all $m \in \mN$, the matrix $\ds \big\{\big([c_0\mathbf I+B]w_j, w_{\alpha(i)}\big)_{\Omega}\big\}_{i,j=1}^m$ is nonsingular (see Lemma 2.3 in \cite{owc}).
\begin{remark} We note here that in order to construct solutions invoking ODE theory and obtain the subsequent energy estimates below, we {\em require} the initial condition $d_0 \in L^2(\Omega)$; if $d_0 \in V'$, additional information about the continuity, adjoint, and invertibility of $B$ on $V'$ would be needed.\end{remark}

By standard existence theory for ordinary differential equations, there exists a unique, absolutely continuous function $d_n(t) = [ d^k_n(t)]_{k=1}^n$ that solves (\ref{Galerkin2}). Therefore  $p_n(t) \in H^1(0, T; V)$ defined in (\ref{pninVn}) is a solution for (\ref{Galerkin1}) for a.e. $t \in [0,T]$. \\

\noindent \underline{Energy Estimates:} We can interpret \eqref{varform1} a.e. $s \in [0,T]$ and let $ q = p_n \in H^1(0, T; V)$ in (\ref{varform1}) to obtain  
$$\langle[c_0\mathbf I  + B]p'_n(s), p_n(s)\rangle+  A[p_n(s), p_n(s);z(s)] = \langle S(s), p_n(s)\rangle$$
Due to the fact that $B$ is self-adjoint on $L^2(\Omega)$ and $p_n(t) \in H^1(0, T; V)$, we have that 
$$\langle[c_0 \mathbf I  + B]p'_n(s), p_n(s) \rangle = \frac{1}{2}\frac{d}{ds}\big([c_0\mathbf I  + B]p_n(s),p_n(s)\big)_{\Omega}$$
Moreover, with $k_1$ as the lower bound on $k$, i.e., the coercivity parameter for $A$ in Remark (\ref{AssumpA}), we have
$$| \langle S(s), p_n(s)\rangle | \leq \frac{1}{2k_1}\|S(s)\|^2_{V'} + \frac{k_1}{2} \|p_n(s)\|^2_V$$
Thus, with the coercivity assumed in Remark (\ref{AssumpA}), we obtain
$$\frac{1}{2}\frac{d}{ds}\big([c_0\mathbf I  + B]p_n(s),p_n(s)\big)_{\Omega} + \frac{k_1}{2}\|p_n(s)\|^2_V \leq \frac{1}{2k_1}\|S(s)\|^2_{V'}$$
We integrate over $(0,t)$ and obtain
$$\big([c_0 \mathbf I + B]p_n(t),p_n(t)\big)_{\Omega} + k_1 \int_0^t \|p_n(s)\|^2_V \ ds \leq  ([c_0\mathbf I  + B]p_n(0),p_n(0))_{L^2(\Omega)} + \frac{1}{k_1}\int_0^t \|S(s)\|^2_{V'} \ ds$$

Using the properties of the operator $B$ and $B^{1/2}$ (as in Lemma \ref{posop} and the discussion following it), we obtain from this estimates point wise (in time) control of $||p_n(t)||_{L^2(\Omega)}$ and $||B^{1/2}p_n(t)||_{L^2(\Omega)}$ for each $t \in [0,T]$, as well as 
$$p_n,\ B^{1/2}p_n,\ (c_0  + B)^{1/2}p_n \in L^{\infty}(0,T;L^2(\Omega)), \ \ p_n \in L^2(0,T;V).$$ Thus 
$$\|p_n\|_{L^{\infty}(0,T;L^2(\Omega))}^2+||B^{1/2}p||_{L^{\infty}(0,T;L^2(\Omega))}^2 \leq (d_0, ([c_0\mathbf I  + B]^{-1}d_0)_{\Omega} + \frac{1}{k_1}\int_0^t \|S(s)\|^2_{V'} \ ds$$
and 
$$\|p_n\|_{L^2(0,T;V)}^2 \leq (d_0, [c_0\mathbf I  + B]^{-1}d_0)_{\Omega} + \frac{1}{k_1}\int_0^t \|S(s)\|^2_{V'} \ ds$$
Since $B$ is continuous from $V$ into $H^1(\Om)$, we obtain that $Bp_n \in L^2(0,T;H^1(\Om))$, and thus we have  
\begin{equation}\label{estimatezetan}
\|(c_0\mathbf I  + B)p_n(t)\|_{L^2(0,T;H^1(\Om))} \leq C \|p_n(t)\|_{L^2(0,T;V)}
\end{equation}

Now, directly from the (\ref{varform1'}), using the characterization of the norm in $V'=H^{-1}(\Omega)$, we obtain
$$\|[(c_0 \mathbf I + B)p_n]'(s)\|_{V'} \leq M \|p_n(s)\|_{V} + \|S(s)\|_{V'},$$
which implies that
$$[(c_0 \mathbf I  + B)p_n]' \in  L^2(0,T;V')$$ with   
$$\int_0^t \|[(c_0 \mathbf I  + B)p_n]'(s)\|^2_{V'} \leq 2M^2  \int_0^t \|p_n(s)\|^2_{V} + \int_0^t  \|S(s)\|^2_{V'} \lesssim (d, [c_0\mathbf I  + B]^{-1}d)_{\Omega} + \int_0^t  \|S(s)\|^2_{V'} $$

\medskip

\noindent \underline{Existence:} 
Since $\{p_n\}$ is bounded in $L^2(0,T;V)$, we can extract a weakly convergent subsequence $p_{n_k}$. If we call the weak limit $p$,  then we have that
\begin{equation}\label{weakconvgpn}
p_{n_k} \rightharpoonup p \ \ \text{in}\ \ L^2(0,T;V)
\end{equation}
Using the continuity of the operator $B: V \to H^1(\Omega)$, we obtain that 
\begin{equation}\label{weakconvgBpn}
(c_0 \mathbf I + B)p_{n_k} \rightharpoonup (c_0 \mathbf I + B)p \ \ \text{in}\ \ L^2(0,T;H^1(\Om))
\end{equation}
According to the energy estimates above, we have that the subsequence $\{[(c_0+B)p_{n_k}]'\}$ is bounded in $L^2(0,T;V')$. Consequently, we obtain on a new subsequence (retaining the subscript $n_k$) that
\begin{equation}\label{weakconvgBpnprime}
[(c_0+B)p_{n_k}]'  \rightharpoonup [(c_0 + B)p]' \ \ \text{in}\ \  L^2(0,T;V') 
\end{equation}
Now invoking \eqref{varform1} we can write 
\begin{equation}\label{varformforsubseq}
\int_0^T \langle [(c_0\mathbf I+B)p_{n_k}]'(t), q(t) \rangle \ dt + \int_0^T A[p_{n_k}(t), q(t);z(t)]\ dt = \int_0^T \langle S, q \rangle  \ dt
\end{equation}
for every $q \in L^2(0,T; V_{n_k})$. Choose $N$ such that $N \leq n_k$. In \eqref{varformforsubseq},  let $q = w\varphi$, with 
$w \in V_N$ and $\varphi \in \cD(0,T)$, and let $n_k \to \infty$. Thanks to \eqref{weakconvgpn} and \eqref{weakconvgBpnprime} and the continuity  of the bilinear form $A$  we infer that 
\begin{equation}\label{varformforlimitofsubseq}
\int_0^T \Big\{\big\langle[(c_0\mathbf I + B)p]'(t), w\big\rangle + A[p(t), w;z(t)] - \langle S(t), w\rangle\Big\}\varphi(t) \ dt = 0
\end{equation}
Letting $N \to \infty$ and using the fact that $\varphi$ is arbitrary, we obtain that 
$$\big\langle [(c_0\mathbf I + B)p]' (t), w\big\rangle + A[p(t), w;z(t)] = \langle S(t), w\rangle , \ \text{for a.e.}\ t \in (0,T), \ \text{and for all} \ w \in V,$$ from which \eqref{varform1} follows. 

It remains to check that $p$ satisfies the initial condition $[c_0\mathbf I+B]p(0) = d_0$. We use \eqref{varformforlimitofsubseq} with $\varphi \in C^1([0,T])$ that satisfies $\varphi(0) = 1$ and $\varphi(T) =0$, and integrate by parts in the fist term. We obtain
\begin{equation}\label{recoverIC1}
\int_0^T\Big \{-\big([c_0\mathbf I + B)p(t), w\big)_{\Omega}\varphi'(t) + A[p(t), w;z(t)]\varphi(t) - \langle S(t), w\rangle \varphi(t)\Big\} dt = ([c_0\mathbf I+ B]p(0), w)_{\Omega}\end{equation}
Similarly, we use $q(t) = \varphi(t) w$ with $w \in V_n$ in \eqref{varformforsubseq}, and integrate by parts in the first term. We obtain 
\begin{equation}\label{recoverIC2}
\int_0^T\Big \{-\big([c_0\mathbf I+B]p_{n_k}(t), w)_{\Omega}\varphi'(t) + A[p_{n_k}(t), w;z(t)]\varphi(t) - \langle S(t), w\rangle \varphi(t)\Big\} dt = ([c_0\mathbf I+B]p_{n_k}(0), w)_{\Omega}\end{equation}

If we let $n_k \to \infty$ in \eqref{recoverIC2}, the LHS converges to the LHS of \eqref{recoverIC1} due to \eqref{weakconvgBpn}, and the RHS $((c_0\mathbf I+B)p_{n_k}(0), w)_{\Omega} \to (d_0,w)_{\Omega}$. Therefore we obtain that $((c_0\mathbf I+B)p(0), w)_{\Omega} = 
(d_0,w)_{\Omega}$, and using the density of $V$ into $\Omega$ we have that $[c_0\mathbf I+B]p(0) = d_0$ as desired. \\

Finally, we also note that from  \eqref{weakconvgBpn} and \eqref{weakconvgBpnprime} we obtain that 
\begin{equation*}\label{weakconvgBpn*}
(c_0\mathbf I  + B)p_{n_k} \to (c_0\mathbf I  + B)p \ \ \text{in}\ \ L^2(0,T;L^2(\Om))
\end{equation*}
\end{proof}
Note that, through the limit point construction, we obtain the estimates in \eqref{estsfollow} on the constructed solutions by the weak lower semicontinuity of the norm.

\end{document}